\documentclass[final,leqno,onefignum,onetabnum]{siamltex1213}

\usepackage{graphics,graphicx,epstopdf}

\usepackage{longtable}

\usepackage[leqno]{amsmath}
\usepackage{amsxtra, amsfonts,amscd,amssymb,  graphicx,epsf,subfigure}
\usepackage[algo2e,linesnumbered, vlined,ruled]{algorithm2e}

\numberwithin{equation}{section}

\input psfig.sty

\newtheorem{remark}[theorem]{Remark}

\newcommand{\tr}{\mathrm{tr}}

\newcommand{\bX}{\bar X}

\newcommand{\R}{\mathbb{R}}

\newcommand{\C}{\mathbb{C}}

\newcommand{\F}{\mathcal{F}}
\newcommand{\T}{\mathcal{T}}

\newcommand{\xb}{\mathbf{x}}

\newcommand{\half}{\frac{1}{2}}

\newcommand{\be}{\begin{equation}}
\newcommand{\ee}{\end{equation}}
\newcommand{\bee}{\begin{equation*}}
\newcommand{\eee}{\end{equation*}}
\newcommand{\bea}{\begin{eqnarray}}
\newcommand{\eea}{\end{eqnarray}}
\newcommand{\beaa}{\begin{eqnarray*}}
\newcommand{\eeaa}{\end{eqnarray*}}

\usepackage[dvipsnames]{xcolor}
\usepackage[normalem]{ulem}

\newcommand{\revise}[1]{{\color{blue}#1}}

\begin{document}

\title{A Regularized Newton Method for Computing Ground States of Bose-Einstein condensates\thanks{
Part of this work was done when the authors were visiting the Institute for Mathematical Sciences
at the National University of Singapore in 2015.}}
\author{Xinming Wu\thanks{The Key Laboratory of Mathematics
for Nonlinear Sciences, School of Mathematical Sciences, Fudan University, CHINA ({\it wuxinming@fudan.edu.cn}).
Research supported in part by NSFC grants 91330202 and 11301089.}
\and Zaiwen Wen\thanks{Beijing International Center for Mathematical Research, Peking University, CHINA ({\it wenzw@math.pku.edu.cn}).
Research supported in part by NSFC grants 11322109 and 91330202.}
\and Weizhu Bao\thanks{Department of Mathematics and
Center for Computational Science and Engineering, National
University of Singapore, Singapore 119076 ({\it
matbaowz@nus.edu.sg}, URL: http://www.math.nus.edu.sg/\~{}bao/).
Research supported in part  by the Ministry of Education
of Singapore grant R-146-000-196-112.}}

\maketitle

\begin{abstract}
In this paper, we propose a regularized Newton method for  computing ground
states of Bose-Einstein condensates (BECs), which can be formulated as
 an energy minimization problem with a spherical
constraint. The energy functional and constraint
are discretized by either the finite difference, or sine or Fourier
pseudospectral discretization schemes and thus the original infinite dimensional
nonconvex minimization problem is approximated by a finite dimensional constrained
nonconvex minimization problem. Then an initial solution is first constructed by using a feasible gradient type
method, which is an explicit scheme and maintains the spherical constraint automatically.
To accelerate the convergence of the gradient type method, we approximate the energy
functional by its second-order Taylor expansion with a regularized term at each
Newton iteration and adopt a cascadic multigrid technique for selecting initial data.
It leads to a standard trust-region subproblem and we
 solve it again by the feasible gradient type method.
The convergence of the regularized Newton method is established by adjusting
the regularization parameter as the standard trust-region strategy.
Extensive numerical experiments on challenging examples,
including a BEC in three dimensions with an optical lattice potential and
rotating BECs in two dimensions with rapid rotation and strongly repulsive interaction, show that
our method is efficient, accurate and robust.

\end{abstract}

\begin{keywords}
Bose-Einstein condensation, Gross-Pitaevskii equation, ground state,
energy functional, spherical constraint, gradient type method, regularized Newton method.
\end{keywords}

%====================%====================%====================
\section{Introduction}

Since the first experimental realization in dilute bosonic atomic gases~\cite{AEMWC-1995, BSTH-1995, DMAvDKK-1995}, Bose-Einstein condensation (BEC) has attracted great interest in the atomic, molecule and optical (AMO) physics community and condense matter community~\cite{Fetter-2009, Leggett-2001, MAHHWC-1999, RAVXK-2001}.
The properties of the condensate at zero or very low temperature are well described by the nonlinear Schr\"odinger equation (NLSE) for the macroscopic wave function $\psi=\psi(\xb,t)$,
which is also known as the Gross-Pitaevskii equation (GPE) in
three dimensions (3D) ~\cite{Ang-Ket-2002, DGPS-1999, Gross-1961, Pet-Smi-2002, Pitaevskii-1961, Pit-Str-2003} as
\be\label{GPE-original}
i \hbar\frac{\partial\psi(\xb,t)}{\partial t} = \left( -\frac{\hbar^2}{2m}\nabla^2 + V(\xb) + NU_0|\psi(\xb,t)|^2 - \Omega L_z\right)\psi(\xb,t),
\ee
where $t$ is time, $\xb=(x,y,z)^\top\in\R^3$ is the spatial coordinate vector, $m$ is the atomic mass, $\hbar$ is the Planck constant, $N$ is the number of atoms in the condensate, $\Omega$ is an angular velocity, $V(\xb)$ is an external trapping potential.
The term $U_0=\frac{4\pi\hbar^2a_s}{m}$ describes the interaction between atoms in the condensate with the $s$-wave scattering length $a_s$ (positive for repulsive interaction and negative for attractive interaction) and
\[
L_z = xp_y-yp_x = -i\hbar(x\partial_y-y\partial_x)
\]
is the $z$-component of the angular momentum $\mathbf{L} = \xb\times\mathbf{P}$
with the momentum operator $\mathbf{P} = -i\hbar\nabla = (p_x,p_y,p_z)^\top$.
It is also necessary to normalize the wave function properly, i.e.,
\be\label{normalization-original}
\|\psi(\cdot,t)\|^2:=\int_{\R^3} |\psi(\xb,t)|^2 dx = 1.
\ee

By using a proper nondimensionalization and dimension reduction in some limiting
trapping frequency regimes \cite{Bao-Wang-Markowich-2005,Fetter-2009},
we can obtain the dimensionless GPE in $d$-dimensions
($d=1,2,3$ when $\Omega=0$ for a non-rotating BEC and $d=2,3$ when $\Omega\ne0$ for a rotating BEC)
%\cite{Pit-Str-2003,Pet-Smi-2002,Bao-Cai-2013-review}:
\cite{Bao-Cai-2013-review,Pet-Smi-2002,Pit-Str-2003}:
\be\label{GPE}
i \frac{\partial\psi(\xb,t)}{\partial t} = \left(-\frac12 \nabla^2 + V(\xb) + \beta |\psi(\xb,t)|^2 - \Omega L_z\right)\psi(\xb,t), \quad \xb\in\R^d, \ t>0,
\ee
with the normalization condition
\be\label{normalization-Rd}
\|\psi(\cdot,t)\|^2:=\int_{\R^d} |\psi(\xb,t)|^2 d\xb = 1,
\ee
where $\beta\in {\mathbb R}$ is
the dimensionless interaction
coefficient,  $L_z=-i(x\partial_y-y\partial_x)$
and  $V(\xb)$ is a dimensionless real-valued external trapping potential.
In most applications of BEC, the harmonic potential is used \cite{Bao-Du-2004,Bao-Tang-2003}
\be\label{harmpen}
V(\xb)=\frac{1}{2}\left\{\begin{array}{ll}
\gamma_x^2 x^2, &d=1,\\
\gamma_x^2 x^2+\gamma_y^2 y^2, &d=2,\\
\gamma_x^2 x^2+\gamma_y^2 y^2+\gamma_z^2 z^2, &d=3,\\
\end{array}
\right.
\ee
where $\gamma_x$, $\gamma_y$ and $\gamma_z$ are three given positive constants.

Define the energy functional
\be\label{energy}
E(\phi) = \int_{\R^d}\left[\frac12|\nabla\phi(\xb)|^2 + V(\xb)|\phi(\xb)|^2 + \frac{\beta}{2}|\phi(\xb)|^4 - \Omega\bar\phi(\xb)L_z\phi(\xb)\right] d\xb,
\ee
where $\bar{f}$ denotes the complex conjugate of  $f$,
then the  ground state of a BEC is usually defined as the minimizer
of the following nonconvex minimization problem
%\cite{LiebSeiringerPra2000,AD,Pit-Str-2003,Pet-Smi-2002,Bao-Cai-2013-review}:
\cite{AD,Bao-Cai-2013-review,LiebSeiringerPra2000,Pet-Smi-2002,Pit-Str-2003}:
\be\label{prob-min}
\phi_g = {\rm arg\;min}_{\phi \in S} \quad  E (\phi),
\ee
where the spherical constraint $S$ is defined as
\be \label{normg}
S = \left\{ \phi \ |\  E(\phi)<\infty,\ \int_{\R^d} |\phi(\xb)|^2 d\xb = 1  \right\}.
\ee
It can be verified that the first-order optimality condition (or Euler-Lagrange equation)
of \eqref{prob-min} is
the nonlinear eigenvalue problem, i.e., find $(\mu\in{\mathbb R},\phi(\xb))$ such that
\be\label{prob-eig}
\mu\,\phi(\xb)= -\frac12 \nabla^2\phi(\xb) + V(\xb)\phi(\xb) + \beta
|\phi(\xb)|^2\phi(\xb) - \Omega L_z\phi(\xb), \quad \xb\in\R^d,
\ee
with the spherical constraint
\be
\label{normalization}
\|\phi\|^2:=\int_{\R^d} |\phi(\xb)|^2 d\xb =  1.
\ee
Any eigenvalue $\mu$ (or chemical potential in the physics literatures)
of \eqref{prob-eig}-\eqref{normalization} can be computed from its corresponding eigenfunction $\phi(\xb)$ by
%\cite{Pit-Str-2003,Pet-Smi-2002,Bao-Cai-2013-review}
\cite{Bao-Cai-2013-review,Pet-Smi-2002,Pit-Str-2003}
\bee\label{chemical-potential}
\mu =  \mu(\phi) =
 E(\phi) + \int_{\R^d} \frac{\beta}{2}|\phi(\xb)|^4 d\xb.
\eee
In fact, (\ref{prob-eig}) can also be obtained from the GPE \eqref{GPE} by taking the anstaz
$\psi(\xb,t) = e^{-i\mu t}\, \phi(\xb)$, and thus it is also called as time-independent GPE %\cite{Pit-Str-2003,Pet-Smi-2002,Bao-Cai-2013-review}.
\cite{Bao-Cai-2013-review,Pet-Smi-2002,Pit-Str-2003}.

One of the two major problems in the theoretical study of BEC is
to analyze and efficiently compute the ground state $\phi_g$ in
(\ref{prob-min}), which plays an important role in understanding the
theory of BEC as well as predicting and guiding experiments.
For the existence and uniqueness as well as non-existence
of the ground state under different parameter regimes,
we refer to
%\cite{LiebSeiringerPra2000,LiebSeiringer1,Bao-Cai-2013-review}
\cite{Bao-Cai-2013-review,LiebSeiringer1,LiebSeiringerPra2000}
and references therein.
Different numerical methods have been proposed
 for computing the ground state of BEC in
the literatures, which can be classified into two different classes
through different formulations and numerical techniques.
The first class of numerical methods has been designed via the formulation of
 the nonlinear eigenvalue
problem (\ref{prob-eig}) under the constraint (\ref{normalization}) with
different numerical  techniques, such as
the Runge-Kutta type method \cite{Adhikari-2000,Edw-Bur-1995}
for a BEC in 1D and 2D/3D with radially/spherically symmetric external trap,
the simple analytical type method \cite{Dodd-1996},
the direct inversion in the iterated subspace method \cite{Sch-Fed-1999},
the finite element approximation via the Newton's method for solving the
nonlinear system \cite{Bao-Tang-2003},
the continuation method \cite{CC2} and
the Gauss-Seidel-type method \cite{CLS-2005}. In these
numerical methods, the time-independent nonlinear eigenvalue
problem (\ref{prob-eig}) and the constraint (\ref{normalization})
are discretized in space via different numerical methods,
such as finite difference, spectral and finite element methods,
and the ground state is computed numerically via different iterative techniques.
The second class of numerical methods has been constructed via
the formulation of the constrained minimization problem (\ref{prob-min})
with different gradient techniques for dealing with the minimization and/or projection
techniques for handling the spherical constraint, such as
the explicit imaginary-time algorithm used in the physics literatures
%\cite{CCPST-2000,CST-2000,CLS-2005,AD,Aft2003,RHBE-1995},
\cite{AD,Aft2003,CCPST-2000,CLS-2005,CST-2000,RHBE-1995},
the Sobolev gradient method \cite{Garcia34},
the normalized gradient flow method via the backward Euler finite
difference (BEFD) or Fourier (or sine) pseudospectral (BEFP) discretization
method
%\cite{Bao-Du-2004,Bao-2004,Bao-Wang-Markowich-2005,Bao-Chern-Lim-2006,GPELab-2013,GPELab-2014}
\cite{GPELab-2014,GPELab-2013,Bao-2004,Bao-Chern-Lim-2006,Bao-Du-2004,Bao-Wang-Markowich-2005}
which has been
extended to compute ground states of spin-1 BEC
%\cite{Bao-Wang-2007,Bao-Chern-Zhang-2013},
\cite{Bao-Chern-Zhang-2013,Bao-Wang-2007},
dipolar BEC \cite{Bao2010}
and spin-orbit coupled BEC \cite{Bao-Cai-2015},
 and the new Sobolev gradient method \cite{Dan2010}.
In these numerical methods, the time-independent infinitely dimensional
constrained minimization problem (\ref{prob-min}) is first re-formulated to a time-dependent
gradient-type partial different equation (PDE) which is then discretized in space and time
via different discretization techniques and the ground state is obtained
numerically as the steady state of the gradient-type PDE with a proper choice of initial data.

Among those existing numerical methods for computing the ground state
of BEC, most of them converge only linearly in the iteration and/or
require to solve a large-scale linear system per iteration.
Thus the computational cost is quite expensive especially for the large scale problems,
such as the ground state of a BEC in 3D with an optical lattice potential or
a rotating BEC with fast rotation and/or strong interaction.
On the other hand, over the last two decades, some advanced optimization methods
have been developed for computing the minimizers of finite dimensional
nonconvex minimization problems, such as the Newton method via
trust-region strategy \cite{ConnGouldTointTrustRegionBook,NocedalWright06,SunYuan06}
which converges quadratically or super-linearly.
The main aim of this paper is to propose
an efficient and accurate regularized Newton method for computing the ground
states of BEC by integrating proper PDE discretization techniques
and advanced modern optimization methods.
By discretizing the energy functional (\ref{energy}) and the spherical constraint
(\ref{normalization}) with either the finite difference, or sine or Fourier
pseudospectral discretization schemes, we approximate the original infinite dimensional
constrained minimization problem (\ref{prob-min}) by a finite dimensional
minimization problem with a spherical constraint. Then we present an explicit feasible gradient type
optimization method to construct an initial solution,
 which generates new trial points along the gradient on the unit ball so that the
constraint is preserved automatically.
The gradient type method is an explicit iterative scheme and the main costs arise from the assembling of the energy functional and its projected gradient on the manifold.
Although this method often works well on well-posed problems, the convergence of
the gradient type method is often slowed down when some parameters in the energy functional become large,
e.g. $\beta\gg1$ and $\Omega$ is near the fast rotation regime in (\ref{prob-min}).
To accelerate the convergence of the iteration, we propose a regularized Newton type method by
approximating the energy functional via its second-order Taylor expansion with a regularized term
at each Newton iteration with the regularization parameter adjusted via the standard trust-region strategy
\cite{ConnGouldTointTrustRegionBook,NocedalWright06,SunYuan06}.
The corresponding regularized Newton subproblem is a standard trust-region
subproblem which can be solved efficiently by  the gradient type method
  since  it is not necessary to solve the subproblem to a
high accuracy, especially, at the early stage of the  algorithm when a good starting
guess is not available.  Furthermore, the numerical performance of the gradient
method can be  improved by the state-of-the-art acceleration
techniques such as Barzilai-Borwein steps and nonmonotone line
search which guarantees global convergence
\cite{ConnGouldTointTrustRegionBook,NocedalWright06,SunYuan06}.
In addition, we adopt a cascadic multigrid technique \cite{Born}
to select a good starting guess at the finest mesh in the computation,
which significantly reduces the computational cost.
Extensive numerical experiments demonstrate that
our approach can quickly reach the vicinity of an optimal solution and produce
a moderately accurate approximation, even for the very challenging and
difficult cases, such as computing the ground state of a BEC in 3D with an optical lattice potential or
a rotating BEC with fast rotation and/or strong interaction.

The rest of this paper is organized as follows.
Different discretizations of the energy functional and the spherical constraint
via the finite difference, sine and Fourier pseudospectral schemes are
introduced in section \ref{discretization}.
In section \ref{optimization}, we present the gradient type method and the regularized
Newton algorithm for solving the discretized minimization problem with a spherical constraint.
Numerical results are reported in section \ref{numerical} to illustrate the efficiency and accuracy
of our algorithms. Finally, some concluding remarks are given in section \ref{conclusion}.
Throughout this paper, we adopt the standard linear algebra notations. In addition,
given $X \in \C^{m\times n}$,  the operators $\bX$, $X^*$, $\Re(X)$ and $\Im(X)$ denote the complex
 conjugate, the complex conjugate transpose, the real and imaginary parts of $X$, respectively.

%====================%====================%====================
\section{Discretization of the energy functional and constraint} \label{discretization}
In this section, we introduce different discretizations of the energy functional (\ref{energy}) and constraint
(\ref{normalization})  in the constrained minimization problem \eqref{prob-min} and reduce it to a finite dimensional
minimization problem with a spherical constraint.
Due to the external trapping potential, the ground state of \eqref{prob-min}
decays exponentially as $|\xb|\to\infty$ \cite{Bao-Cai-2013-review,LiebSeiringer1,LiebSeiringerPra2000}.
Thus we can truncate the energy functional and constraint
from the whole space $\R^d$ to a bounded computational domain $U$
%(usually taken as an interval $D=[a,b]$
%in 1D, a rectangle $D = [a_1, b_1]\times[a_2, b_2]$ in 2D and a
%box $D = [a_1, b_1]\times[a_2, b_2]\times[a_3, b_3]$ in 3D) --
which is chosen large enough such that the truncation error is negligible
with either homogeneous Dirichlet or periodic boundary conditions.
%\be\label{energyD}
%E(\phi)\approx \int_{D}\left[\frac12|\nabla\phi(\xb)|^2 + V(\xb)|\phi(\xb)|^2 + \frac{\beta}{2}|\phi(\xb)|^4 - %\Omega\bar\phi(\xb)L_z\phi(\xb)\right] d\xb,
%\ee
%with
%\be
%\label{normalizationD}
%\int_{D} |\phi(\xb)|^2 d\xb =  1.
%\ee
We remark here that, from the analytical
results
%\cite{LiebSeiringerPra2000,LiebSeiringer1,Bao-Cai-2013-review},
\cite{Bao-Cai-2013-review,LiebSeiringer1,LiebSeiringerPra2000},
when $\Omega=0$, i.e., a non-rotating BEC, the ground state $\phi_g$ can be taken as
a real non-negative function; and when $\Omega\ne0$, i.e., a rotating BEC,
it is in general a complex-valued function, which will be adopted in our numerical
computations.

\subsection{Finite difference discretization}
Here we present discretizations of (\ref{energy}) and (\ref{normalization}) truncated
on a bounded computational domain $U$ with homogeneous Dirichlet boundary condition
by approximating spatial derivatives via
the second-order finite difference (FD) method and the definite integrals
via the composite trapezoidal quadrature. For simplicity of notation, we
only introduce the FD discretization in 1D. Extensions to 2D and 3D without/with
rotation are straightforward
and the details are omitted here for brevity.

 For $d=1$, we take $U=(a,b)$ as an interval in 1D. Let $h=(b-a)/N$ be the spatial mesh size
with $N$ a positive even integer and denote $x_j=a+jh$ for $j=0,1,\ldots,N$,
and thus $a=x_0<x_1<\cdots<x_{N-1}<x_N=b$ be the equidistant partition of $U$.
Let $\phi_j$ be the numerical approximation of $\phi(x_j)$ for $j=0,1,\ldots,N$ satisfying
$\phi_0=\phi(x_0)=\phi_N=\phi(x_N)=0$  and denote $\Phi=(\phi_1,\cdots,\phi_{N-1})^\top$.
The energy functional (\ref{energy})  with $d=1$ and $\Omega=0$ can be truncated and discretized as
\bea
\nonumber
E(\phi) &\approx& \int_a^b\left[\frac12(\phi'(x))^2 + V(x)\phi(x)^2 +
\frac{\beta}{2}\phi(x)^4\right] dx \nonumber\\
\label{energy-1d}
&=&\sum_{j=0}^{N-1} \int_{x_j}^{x_{j+1}}\left[ -\frac12 \phi(x)\phi''(x) +
V(x)\phi(x)^2 + \frac{\beta}{2}\phi(x)^4\right] dx \nonumber\\
&\approx& h\sum_{j=1}^{N-1}\left[
-\frac12\phi_j\frac{\phi_{j+1}-2\phi_j+\phi_{j-1}}{h^2} + V(x_j)\phi_j^2 +
\frac{\beta}{2}\phi_j^4\right] \nonumber\\
\nonumber
&=& h\sum_{j=0}^{N-1}\frac12\left(\frac{\phi_{j+1}-\phi_j}{h}\right)^2 +
h\sum_{j=1}^{N-1}\left[V(x_j)\phi_j^2 + \frac{\beta}{2}\phi_j^4\right]\\
\label{energy-discrete-FD} &=& h\left[ \Phi^\top A\Phi + \frac{\beta}{2}\sum_{j=1}^{N-1}\phi_j^4 \right]:=
E_h(\Phi), \label{energyFD1}
\eea
where $A=(a_{jk})\in\R^{(N-1)\times(N-1)}$ is a symmetric tri-diagonal matrix with entries
\bee
a_{jk} = \begin{cases}
\frac{1}{h^2} + V(x_j), & j=k,\\
-\frac{1}{2h^2}, & |j- k|=1,\\
0, &\mbox{ otherwise}.
\end{cases}
\eee
Similarly, the constraint (\ref{normalization}) with $d=1$ can be truncated and discretized as
\be
\|\phi\|^2\approx \int_a^b \phi(x)^2dx =
\sum_{j=0}^{N-1} \int_{x_j}^{x_{j+1}} \phi(x)^2dx \approx h \sum_{j=1}^{N-1} \phi_j^2:=\|\Phi\|_h^2=  1,
\ee
which immediately implies that the set $S$ can be discretized as
\be
S_h = \left\{ \Phi\in\R^{N-1} \ \left|\right.\  E_h(\Phi)<\infty, \; \|\Phi\|_h^2 = 1  \right\}.
\ee
Hence, the original problem (\ref{prob-min}) with $d=1$
can be approximated by the discretized minimization problem via the FD discretization:
\be\label{prob-min-discrete-FD}
\Phi_g = {\rm arg\;min}_{\Phi\in S_h} E_h(\Phi).
\ee

 Denote $G_h = \nabla E_h(\Phi)$ be the gradient of $E_h(\Phi)$, notice \eqref{energy-discrete-FD},
we have
\be
G_h:=\nabla E_h(\Phi) = 2h\left(A\Phi + \beta \Phi^3 \right),
\ee
where $\Phi^3\in \R^{N-1}$ is defined component-wisely as $ (\Phi^3)_j = \phi_j^3$ for $j=1,\ldots,N-1$.
We remark here that, when the FD discretization is applied,
the matrix $A$ is a symmetric positive definite sparse matrix.
In addition, for the analysis of convergence and second order convergence rate
of the above FD discretization, we
refer the reader to~\cite{Can2010,Zhou-2004}.

%====================%====================
\subsection{Sine pseudospectral discretization}
For a non-rotating BEC, i.e. $\Omega=0$, when high precision is required such as BEC
with an optical lattice potential,
we can replace the FD discretization by the sine pseudospectral (SP) method
when homogeneous Dirichlet boundary conditions are applied.
Again, we only present the discretization in 1D, and extensions to 2D and 3D without rotation
are straightforward and the details are omitted here for brevity.

For $d=1$, using similar notations as the FD scheme, similarly to
(\ref{energyFD1}),  the energy functional (\ref{energy})
with $d=1$ and $\Omega=0$ truncated on $U$  can be
discretized by the SP method as
\be\label{energy-discrete-SP-1}
E(\phi) \approx h\sum_{j=1}^{N-1}\left[
-\frac12\phi_j\;\partial_{xx}^s\phi|_{x=x_j} + V(x_j)\phi_j^2 +
\frac{\beta}{2}\phi_j^4 \right],
\ee
where $\partial_{xx}^s$ is the sine pseudospectral differential operator
approximating the operator $\partial_{xx}$, defined as
\be\label{spdo1}
\partial_{xx}^s\phi|_{x=x_j} = - \sum_{l=1}^{N-1}\lambda_l^2\;\tilde\phi_l\,\sin\left(\frac{jl\pi}{N}\right),\quad j=1,2,\cdots,N-1,
\ee
with $\{\tilde\phi_l\}_{l=1}^{N-1}$
the coefficients of the discrete sine transform (DST) of $\Phi\in \R^{N-1}$, given as
\be\label{tpll}
\tilde\phi_l = \frac{2}{N}\sum_{j=1}^{N-1}\phi_j\sin\left(\frac{jl\pi}{N}\right),
\qquad \lambda_l = \frac{\pi l}{b-a}, \qquad l=1,2,\cdots,N-1.
\ee
Introduce $\mathrm{V}=\diag(V(x_1), \cdots, V(x_{N-1}))$,
$\Lambda=\diag(\lambda_1^2, \ldots,\lambda_{N-1}^2)$ and
$C = (c_{jk})\in \R^{(N-1)\times(N-1)}$ with entries
$c_{jk} = \sin\left(\frac{jk\pi}{N}\right)$ for $j,k=1,\ldots,N-1$
and denote $\widetilde{\Phi}=\left(\tilde{\phi}_1,\ldots,\tilde{\phi}_{N-1}\right)^\top
=\frac{2}{N}C\Phi$.
Plugging (\ref{spdo1}) and (\ref{tpll}) into (\ref{energy-discrete-SP-1}),
we get
\bea\label{energy-discrete-SP-34}
E(\phi) \approx h\left[\Phi^\top B\Phi + \frac{\beta}{2}\sum_{j=1}^{N-1}\phi_j^4\right]:=E_h(\Phi),
\eea
where $B\in \R^{(N-1)\times(N-1)}$ is a symmetric positive definite matrix defined as
\be
B=\frac{1}{N}C\Lambda C + \mathrm{V}.
\ee
In fact, the first term in \eqref{energy-discrete-SP-34} can be computed
efficiently at cost $O(N\ln N)$ through DST as
\be\label{BCD}
\Phi^\top B\Phi
= \frac{N}{4}\widetilde{\Phi}^\top \Lambda \widetilde{\Phi}+\Phi^\top V\Phi= \frac{N}{4}\sum_{l=1}^{N-1}\lambda_l^2\tilde\phi_l^2+\sum_{j=1}^{N-1}V(x_j) \phi_j^2.
\ee
Again, the original problem (\ref{prob-min}) with $d=1$
can be approximated by the discretized minimization problem via the SP discretization:
\be\label{prob-min-discrete-SPp}
\Phi_g = {\rm arg\;min}_{\Phi\in S_h} E_h(\Phi).
\ee

Noticing (\ref{energy-discrete-SP-34}), we have
\be\label{Gh1dsin}
G_h := \nabla E_h(\Phi)=2h \left(B\Phi+ \beta \Phi^3\right)=
2h\left(\frac{1}{N}C\Lambda C\Phi+ \mathrm{V}\Phi + \beta \Phi^3\ \right).
\ee
%\be
%G_j = \left( \lambda_l^2 \hat X_l \right)^{{}_{\displaystyle\check{}}} + 2V(x_j)X_j + \frac{2\beta}{h}X_j^3,\quad j=1,\cdots,N-1,
%\ee
%where
%\be
%\check\phi_j = \frac{2}{N} \sum_{l=1}^{N-1}\phi_l\sin\left(\frac{jl\pi}{N}\right),\quad j=1,\cdots,N-1.
%\ee

%\comm{ Check: X is real. }

%====================%====================
\subsection{Fourier pseudospectral discretization}

For a rotating BEC, i.e. $\Omega\ne0$, due to the appearance of the
angular momentum rotation, we usually truncate
the energy functional (\ref{energy}) and constraint (\ref{normalization})
on a bounded computational domain $U$ with periodic boundary conditions
and approximate spatial derivatives via
the Fourier pseudospectral  (FP) method and the definite integrals
via the composite trapezoidal quadrature. For simplicity of notation, we
only introduce the FP discretization in 2D. Extensions to 3D are straightforward
and the details are omitted here for brevity.

 For $d=2$, we take $U=[a_1,b_1]\times[a_2,b_2]$ as a rectangle in 2D. Let
$h_1=\frac{b_1-a_1}{N_1}$ and $h_2=\frac{b_2-a_2}{N_2}$ be the spatial mesh sizes
with $N_1$ and $N_2$ two positive integers and denote $x_j=a_1+jh_1$ for $j=0,1,\ldots,N_1$, $y_k=a_2+kh_2$ for $k=0,1,\ldots,N_2$. Denote $h=\max\{h_1,h_2\}$ and $U_{jk}=(x_j,x_{j+1})\times (y_k,y_{k+1})$.
Let $\phi_{jk}$ be the numerical approximation of $\phi(x_j,y_k)$ for $j=0,1,\ldots,N_1$ and $k=0,1,\ldots,N_2$
satisfying $\phi_{jN_2}=\phi_{j0}$ for $j=0,1,\ldots,N_1$ and $\phi_{N_1k}=\phi_{0k}$ for $k=0,1,\ldots,N_2$
and denote $\Phi=(\phi_{jk})\in {\mathbb C}^{N_1\times N_2}$.
The energy functional (\ref{energy})  with $d=2$ can be truncated and discretized as
\bea
E(\phi)&\approx&\int_{a_1}^{b_1}\int_{a_2}^{b_2}\left[ -\frac12\bar\phi\Delta\phi + V(x,y)|\phi|^2 + \frac{\beta}{2}|\phi|^4
+  i\Omega\bar\phi(x\partial_y-y\partial_x)\phi  \right] dxdy \nonumber\\
&=&\sum_{j=0}^{N_1-1}\sum_{k=0}^{N_2-1}\int_{U_{jk}}
\left[ -\frac12\bar\phi\Delta\phi + V(x,y)|\phi|^2 + \frac{\beta}{2}|\phi|^4
+  i\Omega\bar\phi(x\partial_y-y\partial_x)\phi  \right] dxdy \nonumber\\
\label{energy-discrete-FP-1}
&\approx& h_1h_2 \sum_{j=0}^{N_1}\sum_{k=0}^{N_2} \left[
-\bar\phi_{jk}\left(\frac12\left.\partial^f_{xx}\phi\right|_{jk}+
\frac12\left.\partial^f_{yy}\phi\right|_{jk}+i\Omega y_k\left.\partial^f_{x}\phi\right|_{jk}- i\Omega x_j\left.\partial^f_{y}\phi\right|_{jk}\right)
\right.\nonumber \\
&& \left. + V(x_j, y_k)|\phi_{jk}|^2 + \frac{\beta}{2}|\phi_{jk}|^4  \right]\alpha_{jk}:=E_h(\Phi),
\eea
where
\[\alpha_{jk}=\left\{\begin{array}{ll}
1 &1\le j\le N_1-1, \ 1\le k\le N_2-1,\\
1/4 &j=0\&k=0,N_2\ \hbox{or} \ j=N_1\&k=0,N_2,\\
1/2 &\hbox{otherwise}, \\
\end{array}\right.
\]
and the Fourier pseudospectral differential operators are given as
\be\label{fftp35}
\begin{split}
\left.\partial^f_x\phi\right|_{jk} =
\sum_{p=-N_1/2}^{N_1/2-1}i\lambda_p\tilde\phi^{(1)}_{pk} e^{i\frac{2\pi jp}{N_1}},\quad
\left.\partial^f_{xx}\phi\right|_{jk} = -
\sum_{p=-N_1/2}^{N_1/2-1}\lambda_p^2\tilde\phi^{(1)}_{pk} e^{i\frac{2\pi jp}{N_1}},\\
\left.\partial^f_y\phi\right|_{jk} =
\sum_{q=-N_2/2}^{N_2/2-1}i\eta_q\tilde\phi^{(2)}_{jq} e^{i\frac{2\pi kq}{N_2}},\quad
\left.\partial^f_{yy}\phi\right|_{jk} = -
\sum_{q=-N_2/2}^{N_2/2-1}\eta_q^2\tilde\phi^{(2)}_{jq} e^{i\frac{2\pi kq}{N_2}},
\end{split}
\ee
with
\be\label{fftp36}
\begin{split}
&\tilde\phi^{(1)}_{pk} = \frac{1}{N_1}\sum_{j=0}^{N_1-1}\phi_{jk}e^{-i\frac{2\pi jp}{N_1}},
\qquad\lambda_p = \frac{2\pi p}{b_1-a_1},\qquad p=-\frac{N_1}{2},\ldots,\frac{N_1}{2}-1,\\
&\tilde{\phi}^{(2)}_{jq} =\frac{1}{N_2}\sum_{k=0}^{N_2-1}\phi_{jk}e^{-i\frac{2\pi kq}{N_2}}, \qquad
\eta_q =\frac{2\pi q}{b_2-a_2},\qquad q=-\frac{N_2}{2},\ldots,\frac{N_2}{2}-1.
\end{split}
\ee
Plugging (\ref{fftp35}) and (\ref{fftp36}) into (\ref{energy-discrete-FP-1}),
the discretized energy functional $E_h(\Phi)$ can be computed efficiently
 via the fast Fourier transform (FFT) as
\bea
\label{energy-discrete-FP}
E_h(\Phi)&=&h_1h_2\left[ \sum_{k=0}^{N_2}\alpha_{1k}N_1
\sum_{p=-N_1/2}^{N_1/2-1}\left(\frac{\lambda_p^2}{2}+y_k\lambda_p\Omega \right)|\tilde\phi^{(1)}_{pk}|^2
 \right. \nonumber \\
&& \qquad +\left. \sum_{j=0}^{N_1}\alpha_{j1}N_2
\sum_{q=-N_2/2}^{N_2/2-1}\left(\frac{\eta_q^2}{2}-x_j\eta_q\Omega \right)|\tilde{\phi}^{(2)}_{jq}|^2
 \right] \nonumber \\
&&+h_1h_2\sum_{j=0}^{N_1}\sum_{k=0}^{N_2}\alpha_{jk}\left[V(x_j, y_k)|\phi_{jk}|^2 + \frac{\beta}{2}|\phi_{jk}|^4\right].
\eea
Similarly, the constraint (\ref{normalization}) with $d=2$ can be truncated and discretized as
\be
\|\phi\|^2\approx \int_{a_1}^{b_1}\int_{a_2}^{b_2}|\phi(x,y)|^2dxdy
\approx h_1h_2 \sum_{j=0}^{N_1-1}\sum_{k=0}^{N_2-1} |\phi_{jk}|^2:=\|\Phi\|_h^2=  1,
\ee
which immediately implies that the set $S$ can be discretized as
\be
S_h = \left\{ \Phi\in\C^{N_1\times N_2} \ \left|\right.\  E_h(\Phi)<\infty, \; \|\Phi\|_h^2 = 1  \right\}.
\ee
Hence, the original problem (\ref{prob-min}) with $d=2$
can be approximated by the discretized minimization problem via the FP discretization:
\be\label{prob-min-discrete-FP5}
\Phi_g = {\rm arg\;min}_{\Phi\in S_h} E_h(\Phi).
\ee
Noticing (\ref{energy-discrete-FP}), similarly to (\ref{Gh1dsin}),
 $G_h = \nabla E_h(\Phi)$ can be computed efficiently
via FFT in a similar manner with the details omitted here for brevity.

%====================%====================%====================
\section{A regularized Newton method by trust-region type techniques}\label{optimization}

It is easy to see that the constrained minimization problems
\eqref{prob-min-discrete-FD}, \eqref{prob-min-discrete-SPp} and
\eqref{prob-min-discrete-FP5} can be written in a unified way via a proper rescaling
\be \label{eq:prob}
X_g:={\rm arg\;min}_{X \in S_M} \F(X):= \half X^* A X + \alpha
\sum_{j=1}^M |X_j|^4, \ee
where $M$ is a positive integer, $\alpha$ is a given real constant,
$A \in \C^{M\times M}$ is a Hermitian matrix and the spherical constraint is given as
\[S_M=\left\{ X=(X_1,X_2,\ldots,X_M)^\top\in \C^M\ |\ \|X\|_2^2:=\sum_{j=1}^M|X_j|^2=1\right\}.
\]

We first derive the optimality conditions of
the problem \eqref{eq:prob}. The gradient and Hessian of $\F(X)$ can be written
 explicitly.
\begin{lemma} \label{eq:grad-Hess-BEC} The first and second-order directional
   derivatives of $\F(X)$ along a direction $D \in \C^M$ are:
   \bea \nabla \F(X)[D] &=&  \Re(D^*AX) + 4 \alpha \sum_{j=1}^M (\bar{X}_j X_j) \Re(\bar{X}_j
   D_j), \\
   \nabla^2 \F(X)[D,D]&=& D^*AD + 4\alpha \sum_{j=1}^M \left[ (\bar{X}_j X_j) (\bar{D}_j
   D_j) + 2 \Re(\bar{X}_j D_j)^2 \right].
   \eea
   \end{lemma}
Define the Lagrangian function of \eqref{eq:prob} as
\be\label{lagrangian}
L(X, \theta) = \F(X) - \frac{\theta}{2}(\|X\|_2^2-1),
\ee
then the first-order optimality conditions of \eqref{eq:prob} are
\bea\label{optim1}
G - \theta X = 0,\\
\|X\|_2 = 1,\label{optim2}
\eea
where $G = \nabla \F(X)$ is the gradient of $\F(X)$.
Multiplying both sides of \eqref{optim1} by $X^*$ and using \eqref{optim2},
we have $\theta=X^* G$. Therefore, \eqref{optim1} becomes
\be\label{optim-cond}
(I-XX^*) G = \mathcal{A}(X)X = 0,  \qquad \hbox{with} \qquad
\mathcal{A}(X) = GX^* - XG^*.
\ee
By definition, $\mathcal{A}(X)$ is skew-symmetric at every $X$.

By differentiating both sides of $X^* X= 1$, we obtain the tangent vector set of the constraints:
\be \label{eq:tangent-St} \T_X :=\{ Z \in \C^M: X^* Z  = 0 \}.  \ee
The second-order optimality conditions is described as follows.
\begin{lemma} \label{lemma:sec-opt}
 1) (Second-order necessary conditions, Theorem 12.5 in \cite{NocedalWright06})
 Suppose that $X \in \C^M$ is a local minimizer of the problem
 \eqref{eq:prob}. Then $X$ satisfies
 \be \label{eq:sec-opt} \nabla^2 \F(X)[D,D] - \theta
 D^* D \ge 0,\quad \forall D \in \T_X,\quad \mbox{where } \theta = \nabla
 \F(X)^* X. \ee
% where $\Lambda = G^\top X$, .

2) (Second-order sufficient conditions, Theorem 12.6 in \cite{NocedalWright06}) Suppose that for   $X \in
\C^M$, there exists a Lagrange multiplier $\theta$ such that the first--order conditions are satisfied. Suppose also    that
 \be \label{eq:suff-sec-opt}\nabla^2 \F(X)[D,D] - \theta
 D^* D  > 0, \ee
  for any vector $D \in\T_X$. Then $X$ is a
  strict local minimizer for \eqref{eq:prob}.
\end{lemma}

%====================%====================
\subsection{Construct initial solutions using feasible gradient type methods}
\label{sec:grad-OptM}
In this subsection, we consider to solve  the problem
\eqref{eq:prob} by following the
feasible method proposed in \cite{OptM-Wen-Yin-2010}.
 The description of the algorithm is included to keep the exposition as
 self-contained as possible. Observe that $\mathcal{A}(X)X$ is the gradient of
 $\F(X)$ at $X$ projected to the tangent space of the constraints. The steepest
 descent path is $\hat Y(\tau):=X-\tau \mathcal{A}(X) X$, where $\tau$ is a positive constant
 representing the step size.  However, this $\hat Y(\tau)$ does not generally have a unit norm.

An alternative implicit updating path is
\be\label{eq-Y}
Y(\tau):= X - \tau \mathcal{A}(X) (X + Y(\tau)) \Longleftrightarrow
Y(\tau) = \left(I + \tau\mathcal{A}(X)\right)^{-1} \left(I -
\tau\mathcal{A}(X)\right) X.
\ee
Then the fact that $\left(I + \tau \mathcal{A}(X)\right)^{-1} \left(I - \tau
\mathcal{A}(X) \right)$ is orthogonal for any $\tau \ge 0$ gives $\|Y(\tau)\|_2 = \|X\|_2 = 1$,
 i.e., the constraints are preserved at every $\tau$. %Since the equation \eqref{eq-Y} is linear with respect to $Y(\tau)$,
 The closed-form solution of $Y(\tau)$ can be computed explicitly as a linear
combination of $X$ and $G$, in which the linear coefficients are determined by $\tau$, $\|X\|_2 $,  $\|G\|_2$ and $X^* G$.

\begin{theorem}\label{thm-Y}
For every $\tau\ge0$,  $Y(\tau)$ of \eqref{eq-Y}
satisfies $\|Y(\tau)\|_2=\|X\|_2$. In addition, $Y(\tau)$ is given in the closed-form
as
\be
Y(\tau)= \alpha(\tau) X + \beta (\tau) G,
\ee
where
\[
\alpha(\tau) = \frac{ \left( 1 +  \tau X^* G \right)^2
- \tau^2 \|X\|_2^2 \|G\|_2^2 }
{ 1- \tau^2 (X^* G)^2
+ \tau^2 \|X\|_2^2 \|G\|_2^2 }, \quad
\beta(\tau)  =   \frac{ - 2\tau  \|X\|_2^2 }
{ 1- \tau^2 (X^* G)^2
+ \tau^2 \|X\|_2^2 \|G\|_2^2 }.
\]
\end{theorem}

We refer to~\cite{OptM-Wen-Yin-2010} for the details of the proof of this theorem.

A suitable step size $\tau$ can be chosen by using a nonmonotone
curvilinear (as our search path is on the manifold rather
than a straight line) search with an initial step size determined by the
Barzilai-Borwein (BB) formula \cite{BarzilaiBorwein1988}.
They were developed originally for the vector case in \cite{BarzilaiBorwein1988}.
At iteration $k$, the step size is computed as
 \be  \label{eq:bb-1} \tau^{k,1} = \frac{\tr\left((S^{(k-1)})^{*}S^{(k-1)}\right)}
    {|\tr\big((S^{(k-1)})^{*}  W^{(k-1)} \big)|} \quad \mbox{ or  } \quad
    \tau^{k,2} = \frac{|\tr\left((S^{(k-1)})^{*} W^{(k-1)}
    \right)|}{ \tr \big((W^{(k-1)})^{*} W^{(k-1)} \big)},\ee
where
$ S^{(k-1)} = X^{(k)} - X^{(k-1)}$ and $W^{(k-1)} = \mathcal{A}(X^{(k)})X^{(k)}
-\mathcal{A}(X^{(k-1)})X^{(k-1)}$. When  $\tau^{k,1}$ or $\tau^{k,2}$ is not bounded, they
are reset to a finite number.

In order to guarantee convergence, the final value for $\tau^{(k)}$ is a fraction
 of $\tau^{k,1}$ or $\tau^{k,2}$ determined by a
nonmonotone search condition. Let $Y(\tau)$ be defined by  \eqref{eq-Y}, $C^{(0)}=\F(X^{(0)})$, $ Q^{(k+1)} = \eta Q^{(k)} +1$ and  $Q^{(0)}=1$. The new points
 are generated iteratively in the form  $X^{(k+1)}:=Y^{(k)}(\tau^{(k)})$ with
  $\tau^{(k)} =\half \tau^{k,1} \delta^m$ or $\tau^{(k)} =\half \tau^{k,2} \delta^m$.
  Here $m$ is the smallest nonnegative integer satisfying
 \be \label{eq:NMLS-Armijo}
 \F(Y^{(k)}(\tau^{(k)})) \le C^{(k)} - \rho_1  \tau^{(k)} \|\mathcal{A}(X^{(k)})X^{(k)}\|_2^2, \ee
 where each reference
 value $C^{(k+1)}$   is taken to be the convex combination of  $C^{(k)}$ and
 $\F(X^{(k+1)})$ as $C^{(k+1)} = (\eta Q^{(k)} C^{(k)} +
            \F(X^{(k+1)}))/Q^{(k+1)}$.
In Algorithm \ref{alg:ConOptM} below, we specify our method for
solving the constrained minimization problem (\ref{eq:prob}) obtained
from the discretization of the ground state of BEC.
Although several backtracking steps may be needed to update the
$X^{(k+1)}$, we observe that the BB step size  $\tau^{k,1}$ or $\tau^{k,2}$ is
often sufficient for \eqref{eq:NMLS-Armijo} to hold in most of our numerical
experiments.

\begin{algorithm2e}[h]\caption{A feasible gradient method}
\label{alg:ConOptM}Given $X^{(0)}$, set $ \rho_1, \eta \in (0,1)$, $k=0$.\\
\While{stopping conditions are not met}{
Compute $\tau^{(k)} \gets \half \tau^{k,1} \delta^m$ or $\tau^{(k)} \gets \half \tau^{k,2} \delta^m$, where
$m$ is the smallest nonnegative integer satisfying the condition \eqref{eq:NMLS-Armijo}.\\
Set $X^{(k+1)}\gets Y(\tau)$. \\
{$Q^{(k+1)}\gets\eta Q^{(k)} +1$ and $C^{(k+1)}\gets(\eta Q^{(k)} C^{(k)}
+ \F(X^{(k+1)}))/Q^{(k+1)}$}.\\
  $k\gets k+1$.
}
\end{algorithm2e}

We can establish the
  convergence of Algorithm \ref{alg:ConOptM} as follows.
\begin{theorem}\label{theo:convergence}
  Let $\{X^{(k)} : k\geq 0\}$ be an infinite sequence generated by the Algorithm
  \ref{alg:ConOptM}. Then either $\|\mathcal{A}(X^{(k)})X^{(k)}\|_2=0$ for some
  finite $k$ or
 \[
\liminf_{k \to \infty}\|\mathcal{A}(X^{(k)})X^{(k)}\|_2 = 0.
\]
\end{theorem}
\begin{proof}
Since the energy function $\F(X)$ is differentiable and its gradient $\nabla
\F(X)$ is Lipschitiz continuous, the results can be obtained using the proofs of
\cite{jiang-dai} in a similar fashion.
\end{proof}

\begin{remark}%Theorem \ref{theo:convergence} only ensures convergence of a subsequence of $\{X^{(k)}\}$.
The convergence of the full sequence $\{X^{(k)}\}$ can be ensured if a monotone
line search is used. Given $\hat \alpha>0, \rho_1, \delta \in (0,1)$, the Armijo point at $X^{(k)}$
is defined as  $Y^{(k)}(\tau^{(k)})$, where
 $Y(\tau)$ is the curve \eqref{eq-Y}, $\tau^{(k)} = \hat \alpha \delta^m$ and
  $m$ is the smallest nonnegative integer satisfying
 \be \label{eq:Armijo1}
 \F(Y^{(k)}(\tau^{(k)})) \le \F(X^{(k)}) - \rho_1  \tau^{(k)} \|\mathcal{A}(X^{(k)}) X^{(k)}\|_2^2. \ee
Using the proofs of Theorem 4.3.1 and Corollary 4.3.2 \cite{opt-manifold-book}
in a similar fashion, we can prove that %establish %the convergence of Algorithm \ref{alg:hybrid} as follows.  %\begin{theorem}\label{theo:convergence} Let $\{X^{(k)} : k\geq 0\}$ be an infinite sequence generated by Algorithm \ref{alg:hybrid}. Then
 $ \lim_{k \to \infty}\|\mathcal{A}(X^{(k)})X^{(k)}\|_2 = 0$.
%\end{theorem}
\end{remark}

%====================%====================
\subsection{A regularized Newton method for computing ground states of BEC}

In general, the Algorithm \ref{alg:ConOptM} works well
in the case of weak interaction and slow rotation, i.e.
 $|\beta|$ and $|\Omega|$ are small in the energy functional \eqref{energy}.
However,  its convergence is often slowed down in the case of strong interaction
and/or fast rotation, i.e., when one of the parameters
becomes larger, and thus it can take a lot of iterations to
obtain a highly accurate solution. Usually, fast local convergence cannot be expected  if only the
gradient information is used, in particular, for difficult non-quadratic problems.
 Observe that the most difficult term in \eqref{eq:prob} is the quartic function
 $|X_i|^4$. A Newton method is to replace $\F(X)$ by its second-order Taylor
 expansion. In order to ensure the global convergence of the  Newton's method, we
 adopt the trust region method  \cite{ConnGouldTointTrustRegionBook,NocedalWright06,SunYuan06}
 by adding a proximal term $\|X-X^{(k)}\|^2_2$ in the
 surrogate function as:
\[
\tilde W^{(k)}(X) := \nabla \F(X^{(k)})[X-X^{(k)}] + \half
\nabla^2 \F(X^{(k)})[X-X^{(k)}, X - X^{(k)}] + \frac{\delta^{(k)}}{2} \|X-X^{(k)}\|_2^2,
\]
where   $\delta^{(k)} > 0$ is a regularization parameter.
Using Lemma \ref{eq:grad-Hess-BEC}, we obtain that
\[
\tilde W^{(k)}(X) =W^{(k)}(X)
+ \mbox{constant},
\]
where
\beaa\label{eq:mL}
W^{(k)}(X) &=&  \half X^*AX  +  4 \alpha \sum_{j=1}^N \left(
\bar{X}^{(k)}_j X^{(k)}_j\right) \Re\left( \bar{X}^{(k)}_j (X_j-X^{(k)}_j)\right)\\
&+& 2\alpha \sum_{j=1}^N \left[ \left( \bar{X}^{(k)}_j X^{(k)}_j + \delta^{(k)} \right)|X_j- X^{(k)}_j|^2 + 2
 \Re\left( \bar{X}^{(k)}_j (X_j- X^{(k)}_j)\right)^2  \right]. \nonumber
\eeaa
The gradient of $W^{(k)}(X)$ is
\[
(\nabla W^{(k)}(X))_j = (AX)_j  + 4 \alpha(
\bar{X}^{(k)}_j X^{(k)}_j) X_j + 8
 \alpha \Re( \bar{X}^{(k)}_j (X_j- X^{(k)}_j))X^{(k)}_j
+ \tau^{(k)}(X_j-X^{(k)}_j).
 \]

We next present the  regularized Newton  framework starting from a feasible
initial point $X^{(0)}$ and the regularization parameter $\delta^{(0)}$.
 At the $k$-th iteration, our regularized Newton subproblem is defined as
\be\label{eq:tr-sub}
\min_{\|X\|_2=1}  \quad W^{(k)}(X)
   \ee
  The subproblem \eqref{eq:tr-sub} is the so-called trust-region subproblem.
  Since the dimension $M$ in (\ref{eq:prob}) is usually very large so that
  the discretization error of (\ref{prob-min}) can be small, the standard
  algorithms for solving the trust-region subproblem
  \cite{ConnGouldTointTrustRegionBook,NocedalWright06,SunYuan06} usually cannot
  be applied to \eqref{eq:tr-sub} directly. Hence, we still use a
  gradient-type method similar to the one described in subsection
  \ref{sec:grad-OptM} to solve \eqref{eq:tr-sub}.
 The method is ideal for solving these  regularized Newton subproblems
  since  it is not necessary to solve these subproblems to a
high accuracy, especially, at the early stage of the  algorithm when a good starting
guess is not available.

 Let  $Z^{(k)}$ be an optimal solution of \eqref{eq:tr-sub}.
 Generally speaking, an algorithm cannot be guaranteed to converge globally if $X^{(k+1)}$ is set
 directly to the
 trial point $Z^{(k)}$ obtained from a model with a fixed $\delta^{(k)}$.
 In order to decide whether the trial point $Z^{(k)}$ should be accepted   and whether the
 regularization parameter should be updated or not, we calculate the ratio
 between the actual reduction of the objective function
   $\F(X)$ and predicted reduction:
  \be \label{alg:TR-ratio} \rho^{(k)} = \frac{ \F(Z^{(k)}) - \F(X^{(k)})  } {
 W^{(k)}(Z^{(k)})-W^{(k)}(X^{(k)}) }.  \ee
   If $\rho^{(k)} \ge \eta_1 > 0$, then the iteration is successful and we set
$X^{(k+1)}= Z^{(k)}$; otherwise, the iteration is not successful and we set $X^{(k+1)}=
X^{(k)}$, that is,
\be \label{eq:update-X}
X^{(k+1)} = \begin{cases} Z^{(k)}, & \mbox{ if } \rho^{(k)} \ge \eta_1, \\
  X^{(k)}, & \mbox{ otherwise}.
\end{cases}
\ee
Then the regularization parameter $\delta^{(k+1)}$ is updated as
\be \label{alg:tau-up} \delta^{(k+1)} \in \begin{cases} (0,\delta^{(k)}], & \mbox{  if }
  \rho^{(k)} > \eta_2, \\ [\delta^{(k)},\gamma_1 \delta^{(k)}], & \mbox{  if } \eta_1 \leq \rho^{(k)}
  \leq \eta_2, \\ [\gamma_1 \delta^{(k)}, \gamma_2 \delta^{(k)}], & \mbox{  otherwise}. \end{cases} \ee
 where $0<\eta _1 \le \eta _2 <1 $ and $1<  \gamma _1 \le \gamma _2 $.
 These parameters  determine how aggressively the
 regularization parameter is decreased when an iteration is successful or it is
 increased when an iteration is unsuccessful.   In practice, the performance of
 the   regularized Newton algorithm is not very sensitive to the  values of the parameters.

The complete  regularized Newton  algorithm to solve  \eqref{eq:prob}  is summarized in
the Algorithm~\ref{alg:TR}.

\begin{algorithm2e}[t]
\caption{A regularized Newton method} %for Computing Ground States of BEC}
\label{alg:TR}
%Initialize: ~~
Given a feasible initial solution $X^{(0)}$ with
$\|X^{(0)}\|_2=1$ and initial regularization parameter $\tau^{(0)}>0$.
 Choose $0<\eta _1 \le \eta _2 <1 $, $1<  \gamma _1 \le \gamma _2 $. \\% and $\alpha, \beta, \gamma \in (0,1)$.
 Call Algorithm \ref{alg:ConOptM} to minimize problem \eqref{eq:prob} to a certain
low accuracy for a feasible solution $X^{(1)}$.  Set iteration $k:=1$.\\
\While{stopping conditions are not met}
{
Solve \eqref{eq:tr-sub} to obtain a new trial point $Z^{(k)}$ . \\
Compute the ratio $\rho^{(k)}$ via \eqref{alg:TR-ratio}. \\
Update $X^{(k+1)}$ from the trial point $Z^{(k)}$ based on \eqref{eq:update-X}. \\
Update $\delta^{(k)}$ according to
\eqref{alg:tau-up}. \\
  $k\gets k+1$.
}
\end{algorithm2e}

 The convergence of the Algorithm
  \ref{alg:TR} can also be
established as follows.
\begin{theorem}   Let $\{X^{(k)} : k\geq 0\}$ be an infinite sequence generated by the Algorithm
  \ref{alg:TR}.  Then either $\|\mathcal{A}(X^{(k)})X^{(k)}\|_2=0$ for some
  finite $k$ or
\[ \lim_{k \to \infty} \|\mathcal{A}(X^{(k)})X^{(k)}\|_2 = 0. \]
\end{theorem}
\begin{proof}
Since the energy function $\F(X)$ is differentiable and its gradient $\nabla
\F(X)$ is Lipschitiz continuous, the results can be obtained using the proofs of
\cite{WenMilzarekUlbrichZhang2013} in a similar fashion.
\end{proof}

%====================
%\section{Practical Issues}
%\subsection{Mesh refinement technique}

%\begin{remark}
The discretization of (\ref{prob-min}) on a fine mesh usually
leads to a problem of huge size ($M\gg1$)
whose computation cost is very expensive, especially for high dimensional case.
A useful technique is to adopt the cascadic multigrid method \cite{Born}, i.e.
solve the minimization problem (\ref{prob-min}) on the coarsest mesh, and then use the obtained solution as the initial guess of the problem on a fine mesh, and repeat until we obtain the solution on the finest mesh.
We present the mesh refinement technique via the cascadic multigrid method
in the Algorithm \ref{alg:meshrefine}, where the discretized problems are
solved from the coarsest mesh to the finest mesh.
%\end{remark}

\begin{algorithm2e}[h]\caption{A cascadic multigrid method for mesh refinement}
\label{alg:meshrefine}%\LinesNumberedHidden
Given an initial mesh $\mathcal{T}^{0}$ and  $X^{(0)}$, set $k=0$.\\
\While{convergence is not met}{
Use $X^{(k)}$ as an initial guess on the $k$th mesh $\mathcal{T}^k$ to calculate
the optimal solution $X^{(k+1)}$ of the minimization problem \eqref{eq:prob}
using the Algorithm  \ref{alg:TR}. \\ %\revise{Algorithm \ref{alg:ConOptM} or \ref{alg:TR}.}\\
Refine the mesh $\mathcal{T}^{k}$ uniformly to obtain $\mathcal{T}^{k+1}$. \\
$k\gets k+1$.
}
\end{algorithm2e}

%====================%====================%====================
\section{Numerical results}\label{numerical}
In this section, we report several numerical examples to illustrate the efficiency and accuracy
of our method. All experiments were performed on a PC with a 2.3GHz CPU (i7 Core)
and the algorithms were implemented in MATLAB (Release 8.1.0).
In our experiments, the Algorithm \ref{alg:ConOptM} is called to compute
the ground state of non-rotating BEC, i.e., $\Omega=0$, since it is a relatively easy
problem. The algorithm is stopped
either when a maximal number of $K$ iterations is reached or  when
\be\label{crit1}
\frac{\|X^{(k+1)} - X^{(k)}\|_\infty}{\tau^{(k)}} \leq \varepsilon_0.
\ee
The default values of $\varepsilon_0$ and $K$ are set to be $10^{-6}$ and $2000$,
respectively. In order to test the spectral accuracy of the SP discretization, a
tighter stopping criterion is taken.
A normalization  step is executed if  $|X^*X-1| > 10^{-14}$
to enforce the feasibility. For non-rotating BEC with strong interaction, i.e., $\beta\gg1$,
the initial solution is usually chosen as the Thomas-Fermi (TF) approximation
\cite{Bao-Cai-2013-review,Bao-Du-2004,Pit-Str-2003}
\be\label{TFA}
\phi_0(\xb) = \begin{cases}
\sqrt{\frac{\mu^{\rm TF}-V(\xb)}{\beta}},& \mbox{ if } V(\xb)\le\mu^{\rm TF},\\
0,&\mbox{ otherwise},
\end{cases}
\ee
where $\mu^{\rm TF}=\frac12\left(\frac{3\beta}{2}\right)^{2/3}$, $\left(\frac{\beta\gamma_y}{\pi}\right)^{1/2}$
and $\frac12\left(\frac{15\beta\gamma_y\gamma_z}{4\pi}\right)^{2/5}$ for $d=1$, $2$ and $3$, respectively.
Since the Algorithm \ref{alg:ConOptM}
may converge slowly for computing the ground state of rotating BEC, i.e., $\Omega\ne0$,
we choose the  regularized Newton method (i.e., Algorithm \ref{alg:TR})
together with the cascadic multigrid method for mesh refinement (i.e., Algorithm \ref{alg:meshrefine})
 and it is terminated when
\be\label{crit2}
\|X^{(k+1)} - X^{(k)}\|_\infty \leq \delta_0,
\ee
where the default value of $\delta_0$ is set to $10^{-8}$.
Let $\phi_g$ be the ``exact'' ground state obtained numerically with a very fine mesh
and we denote its energy and chemical potential as $E_g=E(\phi_g)$ and
$\mu_g=\mu(\phi_g)$, respectively.
To quantify the ground state, one important quantity is the root mean square which is defined as
\be\label{rms}
\alpha_{\rm rms} = \| \alpha\phi_g\|_{L^2(U)} =
\sqrt{\int_{U} \alpha^2|\phi_g(\xb)|^2 d\xb},\quad \alpha=x,y \mbox{ or } z.
\ee

%====================%====================
\subsection{Accuracy test and results in 1D}
We take $d=1$ and $\Omega=0$ in (\ref{prob-min}) and (\ref{energy})
and consider two different trapping potentials

Case I. A harmonic oscillator potential (\ref{harmpen}) with $d=1$, $\gamma_x=1$ and $\beta=400$.

Case II. An optical lattice potential $V(x)=\frac{x^2}{2} + 25\sin^2(\frac{\pi x}{4})$ and $\beta=250$.

The ground state is numerically computed by the Algorithm \ref{alg:ConOptM}
on a  bounded computational domain $U=(-16,16)$ which is partitioned equally with
the mesh size $h$. In order to compare the accuracy of the FD and SP discretizations,
we set $\epsilon_0=10^{-12}$ in (\ref{crit1}).
Let $\phi_{g,h}^{\rm FD}$ and $\phi_{g,h}^{\rm SP}$ be the numerical ground states
obtained with the mesh size $h$ by using FD and SP discretization, respectively.
Table \ref{table-fd1d-sp1d-case1} depicts the numerical errors for Case I,
and respectively, Table \ref{table-fd1d-sp1d-case2} for Case II.

% Tab. 1D, Case-I, Accuracy
\begin{table}[htdp]\caption{Accuracy of the FD and SP discretizations for Case I in \S 4.1.}
\label{table-fd1d-sp1d-case1}
\setlength{\tabcolsep}{8pt}
\begin{center}
\begin{tabular}{lccccc}\hline
Mesh size & $h=1$ &  $h=1/2$  &  $h=1/4$  &  $h=1/8$  \\ \hline
$\max|\phi_g-\phi_{g,h}^{\rm FD}|$ & 2.06E-03 & 1.24E-03 & 2.88E-04  & 7.43E-05  \\ %\hline
$|E_g-E(\phi_{g,h}^{\rm FD})|$ & 8.59E-04 & 2.66E-04 & 6.46E-05 & 1.59E-05 \\ %\hline
$|\mu_g-\mu(\phi_{g,h}^{\rm FD})|$ & 2.21E-02 & 9.48E-05 & 3.49E-05 & 8.60E-06  \\[1mm]  \hline\hline
%Mesh size  &  $h=1$  &  $h=1/2$  &  $h=1/4$  &  $h=1/8$  \\ \hline
$\max|\phi_g-\phi_{g,h}^{\rm SP}|$ & 1.31E-03 & 7.04E-05  & 1.95E-08 & 5.01E-13  \\ %\hline % 4.00E-13
$|E_g-E(\phi_{g,h}^{\rm SP})|$ & 5.69E-05 & 2.64E-06 & 8.45E-12 & 2.17E-13 \\ %\hline % 7.11E-15
$|\mu_g-\mu(\phi_{g,h}^{\rm SP})|$ & 1.66E-02 &  8.71E-05 & 9.55E-10 & 2.52E-12  \\ \hline %1.71E-12
\end{tabular}
\end{center}
\end{table}

% Tab. 1D, Case II, Accuracy
\begin{table}[htdp]\caption{Accuracy of the FD and SP discretizations for Case II in \S 4.1.}
\label{table-fd1d-sp1d-case2}
\setlength{\tabcolsep}{8pt}
\begin{center}
\begin{tabular}{lccccc}\hline
Mesh size  &  $h=1$ & $h=1/2$  &  $h=1/4$  &  $h=1/8$     \\ \hline
$\max|\phi_g-\phi_{g,h}^{\rm FD}|$ & 1.02E-02 & 5.81E-03 & 9.97E-04  & 2.50E-04  \\ %\hline
$|E_g-E(\phi_{g,h}^{\rm FD})|$ & 2.66E-02 & 8.39E-03 & 2.03E-03 & 5.02E-04  \\ %\hline
$|\mu_g-\mu(\phi_{g,h}^{\rm FD})|$ & 1.27E-01 & 4.05E-03 & 8.28E-04 & 2.08E-04  \\[1mm] \hline\hline
%Mesh size  &  $h=1$  &  $h=1/2$  &  $h=1/4$  &  $h=1/8$  \\ \hline
$\max|\phi_g-\phi_{g,h}^{\rm SP}|$ & 7.98E-03 & 1.21E-03  & 2.22E-06 & 1.90E-11  \\ %\hline % 3.50E-13
$|E_g-E(\phi_{g,h}^{\rm SP})|$ & 4.22E-04 & 1.96E-04 & 4.99E-08 & 7.53E-13  \\ %\hline % 2.27E-13
$|\mu_g-\mu(\phi_{g,h}^{\rm SP})|$ & 9.76E-02 &  4.11E-03 & 5.61E-07 & 9.17E-13\\ \hline % 1.29E-12
\end{tabular}
\end{center}
\end{table}

From Tables \ref{table-fd1d-sp1d-case1} and \ref{table-fd1d-sp1d-case2},
it is observed that the SP discretization is spectrally accurate,
while the FD discretization  has only second order accuracy for computing the ground state
of BEC in 1D.
Hence, when high accuracy is required, the SP discretization is preferred since it
needs much fewer grid points, and thus it saves significantly
memory cost and computational cost.

For comparison with existing numerical results in the literatures
%\cite{GPELab-2014,Bao-Cai-2013-review,Bao-Du-2004,Bao-Tang-2003,Bao-Chern-Lim-2006},
\cite{GPELab-2014,Bao-Cai-2013-review,Bao-Chern-Lim-2006,Bao-Du-2004,Bao-Tang-2003},
Figure \ref{fig-sp1d} plots the ground states $\phi_g(x)$  obtained by the SP
discretization for  cases I and II. In addition, their energy, chemical potential
and root mean squares are obtained as for Case I:
$E_g = 21.3601$, $\mu_g = 35.5775$ and $x_{\mathrm{rms}} = 3.7751$;
and for Case II: $E_g = 26.0839$, $\mu_g = 38.0692$ and $x_{\mathrm{rms}} = 3.3609$.
These numerical results agree very well with those reported in the literatures
%\cite{GPELab-2014,Bao-Cai-2013-review,Bao-Du-2004,Bao-Tang-2003,Bao-Chern-Lim-2006}.
\cite{GPELab-2014,Bao-Cai-2013-review,Bao-Chern-Lim-2006,Bao-Du-2004,Bao-Tang-2003}.

% Fig. 1D, Case I-II
\begin{figure}[htdp]\caption{Ground states $\phi_g(x)$ for Case I (left) and Case II (right)
in \S 4.1.}
\centerline{\includegraphics[width=6.5cm,height=4.5cm]{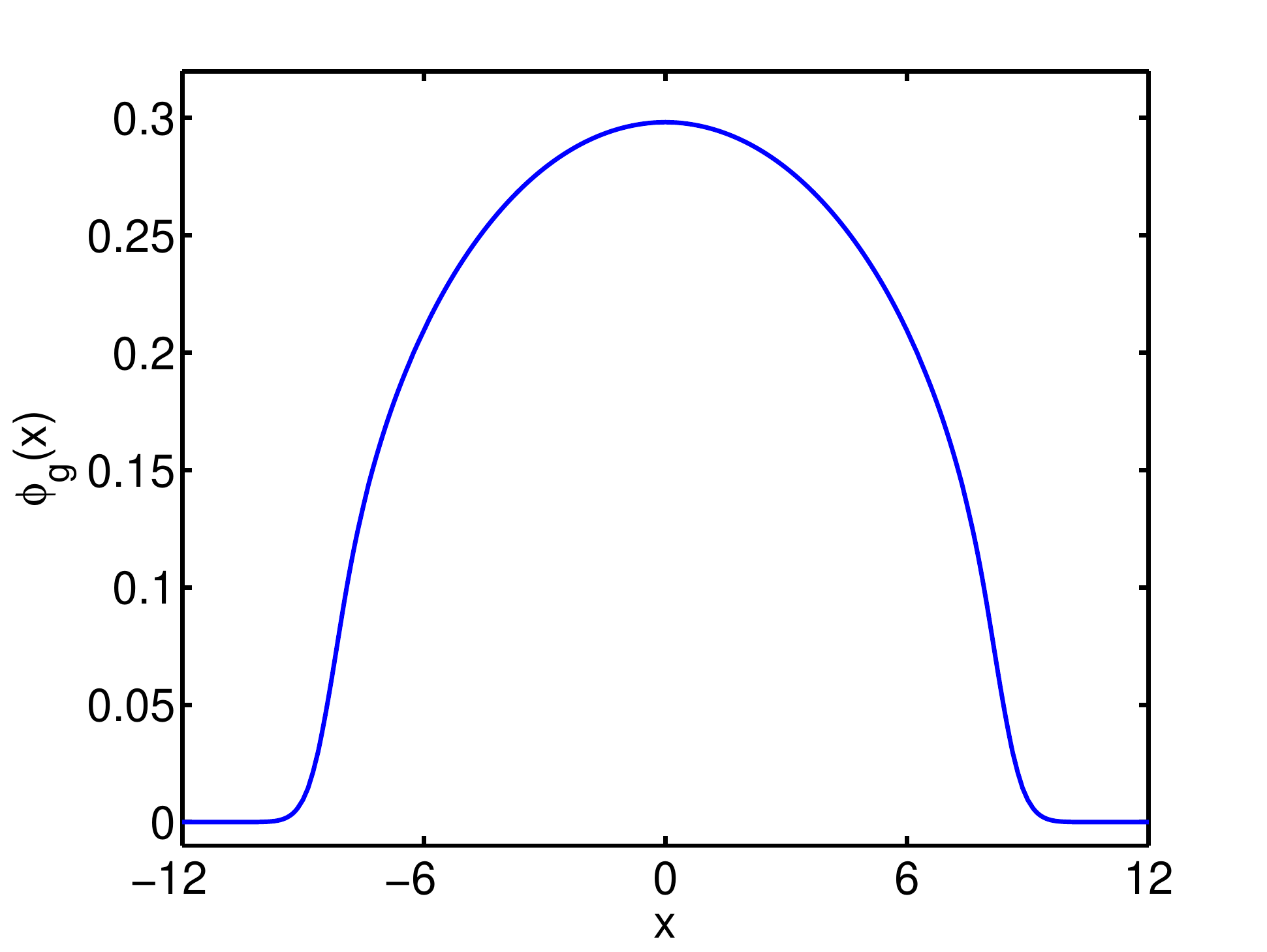}
\includegraphics[width=6.5cm,height=4.5cm]{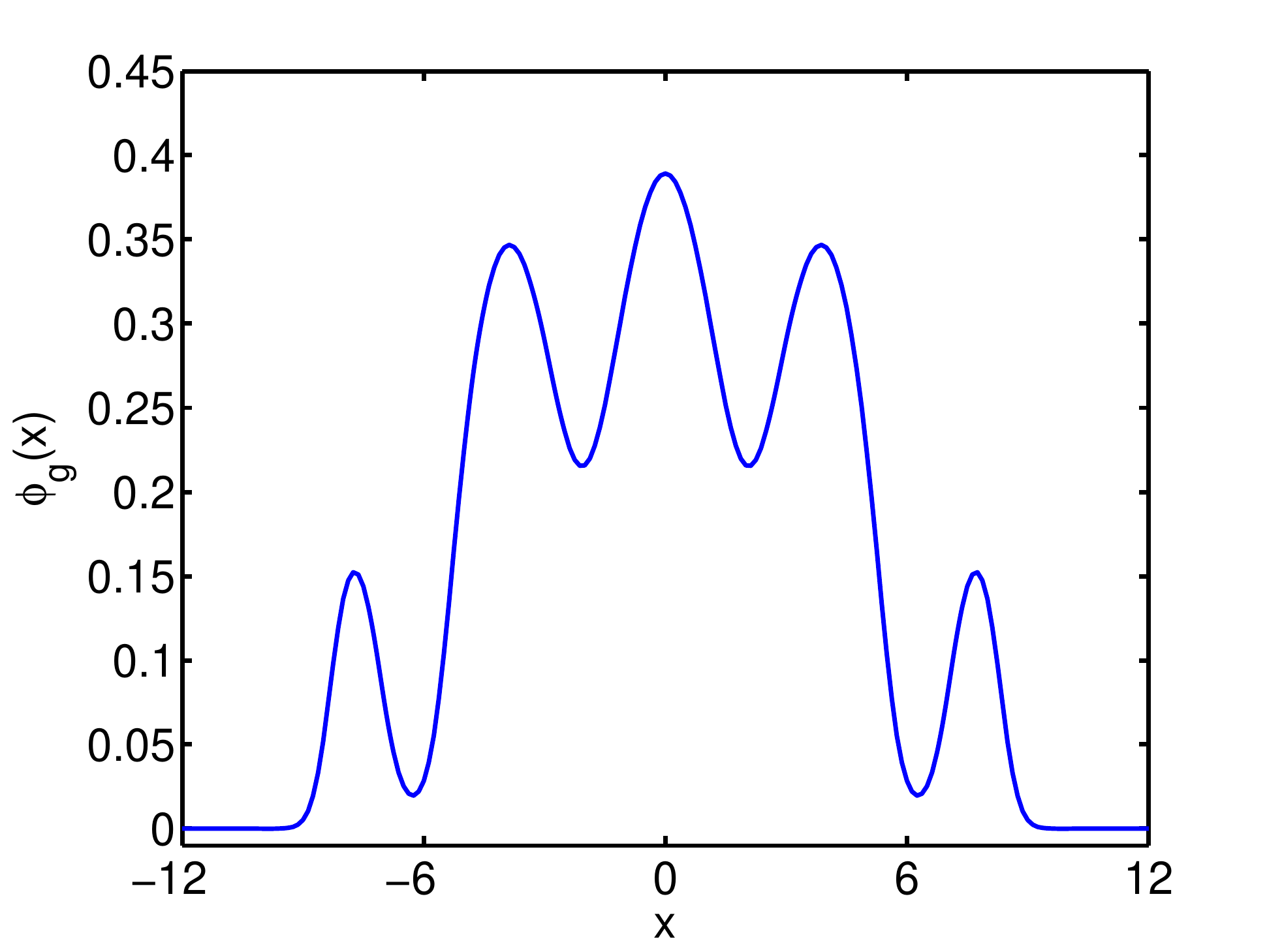}}\label{fig-sp1d}
\end{figure}

%====================%====================
\subsection{Accuracy test and results in 3D}

We take $d=3$ and $\Omega=0$ in (\ref{prob-min}) and (\ref{energy})
and consider two different trapping potentials  \cite{Bao-Du-2004}:

Case I. A harmonic oscillator potential (\ref{harmpen}) with $d=3$, $\gamma_x=1$,
$\gamma_y=2$, $\gamma_z=4$ and $\beta=200$.

Case II. A harmonic oscillator potential and a potential of a
stirrer corresponding to a far-blue detuned Gaussian laser beam
\bee
V(x,y,z) = \frac12(x^2+\gamma_y^2y^2+\gamma_z^2z^2) + \omega_0 e^{-\delta((x-r_0)^2+y^2)}
\eee
with $\gamma_y=1$, $\gamma_z=2$, $\omega_0=4$, $\delta=1$, $r_0=1$ and $\beta=200$.

% Tab. 3D, Ex1, Case-I, Accuracy
\begin{table}[htdp]\caption{Accuracy of the FD and SP discretizations for Case I in \S 4.2.}%\footnotesize
\label{table-fd3d-sp3d-case1}
\setlength{\tabcolsep}{8pt}
\begin{center}
\begin{tabular}{lccccc}\hline
Mesh size  & $h=2$ & $h=1$  &  $h=1/2$  &  $h=1/4$  \\ \hline
$\max|\phi_g-\phi_{g,h}^{\rm FD}|$ & 2.28E-02 & 5.16E-03 & 1.11E-03  & 2.51E-04  \\ %\hline
$|E_g-E(\phi_{g,h}^{\rm FD})|$ & 1.26E-01 & 5.82E-02 & 1.44E-02 & 3.41E-03  \\ %\hline
$|\mu_g-\mu(\phi_{g,h}^{\rm FD})|$ & 4.45E-02 & 3.10E-02 & 9.40E-03 & 2.23E-03  \\[1mm] \hline\hline
%Mesh size  &  $h=2$  &  $h=1$  &  $h=1/2$  &  $h=1/4$ \\ \hline
$\max|\phi_g-\phi_{g,h}^{\rm SP}|$ & 1.10E-02 & 1.68E-03  & 8.68E-06 & 7.34E-10  \\ %\hline % 4.53E-13
$|E_g-E(\phi_{g,h}^{\rm SP})|$ & 1.01E-01 & 6.49E-05 & 1.45E-08 & 1.09E-11  \\ %\hline % 4.68E-12
$|\mu_g-\mu(\phi_{g,h}^{\rm SP})|$ & 1.57E-01 &  4.17E-03 & 5.48E-07 & 1.55E-11  \\ \hline % 1.13E-11
\end{tabular}
\end{center}
\end{table}

Again, the ground state is numerically computed by the Algorithm \ref{alg:ConOptM}
on bounded computational domains $U=(-8,8)\times(-6,6)\times(-4,4)$ and
$U=(-8,8)^3$ for Case I and II, respectively, which are partitioned uniformly with
the same number of nodes in each direction. Let $h$ be the mesh size in the $x$-direction.
Again, we set $\epsilon_0=10^{-12}$ in (\ref{crit1}).
Let $\phi_{g,h}^{\rm FD}$ and $\phi_{g,h}^{\rm SP}$ be the numerical ground states
obtained with the mesh size $h$ by using FD and SP discretization, respectively.
Table \ref{table-fd3d-sp3d-case1} depicts the numerical errors for Case I,
and \revise{respectively}, Table \ref{table-fd3d-sp3d-case2} for Case II.

% Tab. 3D, Ex1, Case II, Accuracy
\begin{table}[htdp]\caption{Accuracy of the FD and SP discretizations for Case II in \S 4.2.}%\footnotesize
\label{table-fd3d-sp3d-case2}
\setlength{\tabcolsep}{8pt}
\begin{center}
\begin{tabular}{lccccc}\hline
Mesh size  & $h=2$ & $h=1$  &  $h=1/2$  &  $h=1/4$   \\ \hline
$\max|\phi_g-\phi_{g,h}^{\rm FD}|$ & 1.61E-02 & 7.92E-03 & 1.69E-03  & 3.92E-04  \\ %\hline
$|E_g-E(\phi_{g,h}^{\rm FD})|$ & 6.76E-01 & 6.06E-02 & 1.33E-02 & 3.16E-03  \\ %\hline
$|\mu_g-\mu(\phi_{g,h}^{\rm FD})|$ & 5.37E-01 & 6.16E-02 & 8.09E-03 & 1.92E-03  \\[1mm] \hline\hline
%Mesh size  &  $h=2$  &  $h=1$  &  $h=1/2$  &  $h=1/4$  \\ \hline
$\max|\phi_g-\phi_{g,h}^{\rm SP}|$ & 1.69E-01 & 2.57E-03 & 4.38E-05 & 1.18E-08 \\ %\hline % 2.11E-13
$|E_g-E(\phi_{g,h}^{\rm SP})|$ & 1.87E-01 & 6.69E-03 & 9.55E-06 & 6.34E-12 \\ %\hline % 6.13E-12
$|\mu_g-\mu(\phi_{g,h}^{\rm SP})|$ & 5.69E-01 &  2.21E-02 & 7.79E-06 & 9.85E-11  \\ \hline % 8.69E-12
\end{tabular}
\end{center}
\end{table}

Again, from Tables  \ref{table-fd3d-sp3d-case1} and \ref{table-fd3d-sp3d-case2},
it is observed that the SP discretization is spectrally accurate,
while the FD discretization  has only second order accuracy for computing the ground state
of BEC in 3D.
Hence, when high accuracy is required and/or the solution
has multiscale phenomena, the SP discretization is preferred since it
needs much fewer grid points, and thus it saves significantly
memory cost and computational cost.

Again, for comparison with existing numerical results in the literatures
%\cite{GPELab-2014,Bao-Cai-2013-review,Bao-Du-2004,Bao-Tang-2003,Bao-Chern-Lim-2006},
\cite{GPELab-2014,Bao-Cai-2013-review,Bao-Chern-Lim-2006,Bao-Du-2004,Bao-Tang-2003},
Figure \ref{fig-sp3d} plots the ground states $\phi_g(x,0,z)$  obtained by the SP
discretization for  cases I and II. In addition, their energy, chemical potential
and root mean squares are obtained as for Case I:
$E_g = 8.3345$, $\mu_g = 11.0102$, $x_{\mathrm{rms}} =
1.6710$,  $y_{\mathrm{rms}} = 0.8751$, and $z_{\mathrm{rms}} =  0.4884$;
and for Case II: $E_g = 5.2696$,  $\mu_g = 6.7019$, $x_{\mathrm{rms}} =
1.3744$, $y_{\mathrm{rms}} = 1.4358$ and $z_{\mathrm{rms}} = 0.7043$.
These numerical results agree very well with those reported in the literatures
%\cite{GPELab-2014,Bao-Cai-2013-review,Bao-Du-2004,Bao-Tang-2003,Bao-Chern-Lim-2006}.
\cite{GPELab-2014,Bao-Cai-2013-review,Bao-Chern-Lim-2006,Bao-Du-2004,Bao-Tang-2003}.

% Fig. 1D, Case I-II
\begin{figure}[htdp]\caption{Ground states $\phi_g(x,0,z)$ for Case I (left) and Case II (right)
in \S 4.2.}
\centerline{\includegraphics[width=6.5cm,height=6.5cm]{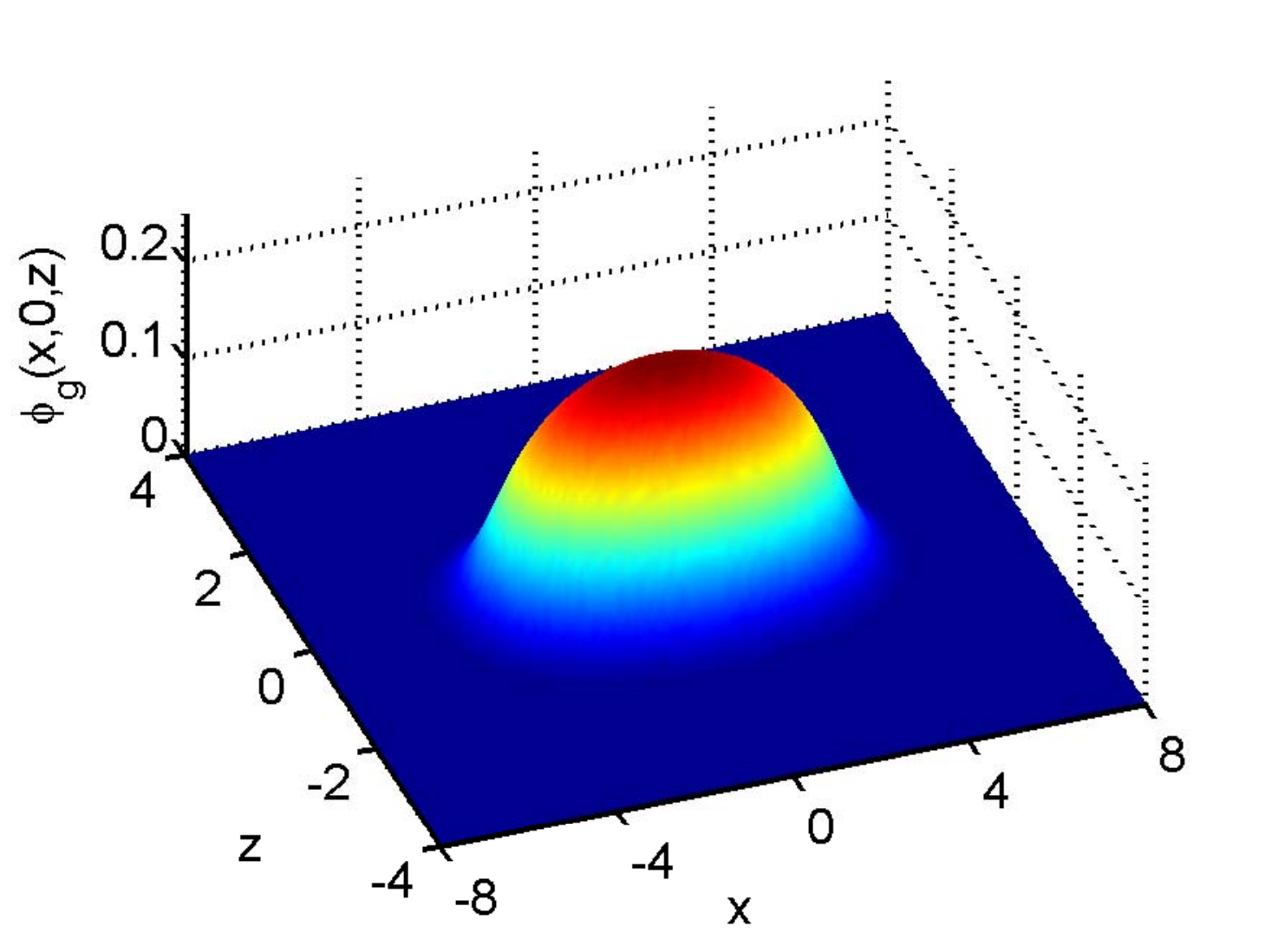}
\includegraphics[width=6.5cm,height=6.5cm]{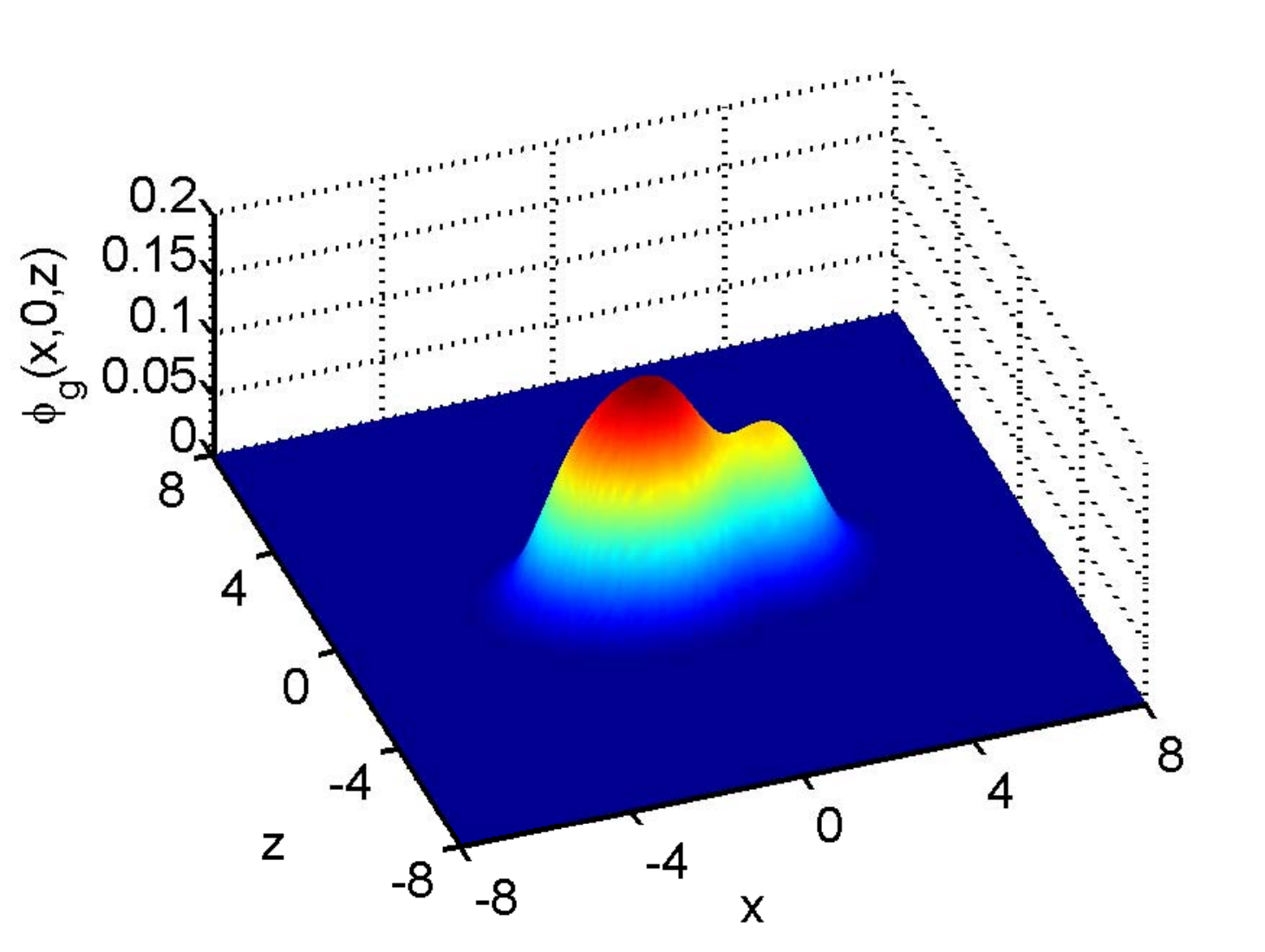}}\label{fig-sp3d}
\end{figure}

To demonstrate the high resolution of the SP
discretization and compare our algorithm with existing numerical
methods \cite{GPELab-2014,Bao-Cai-2013-review,Bao-Chern-Lim-2006}, we also apply our algorithm
to compute the ground state of BEC in 3D with
a combined harmonic and optical lattice potential \cite{Bao-Chern-Lim-2006} as
\be\label{optic98}
V(x,y,z) = \frac12\left(x^2+y^2+z^2\right) + 50\left[ \sin^2\left(\frac{\pi x}{4}\right) + \sin^2\left(\frac{\pi y}{4}\right) + \sin^2\left(\frac{\pi z}{4}\right) \right],
\ee
together with different interaction constants $\beta=100$, $800$ and $6400$.
The ground state is numerically computed by the Algorithm \ref{alg:ConOptM}
on bounded computational domains $U=(-8,8)^3$ for $\beta=100$ and $800$, and $U=(-12,12)^3$ for $\beta=6400$,
which are partitioned uniformly with the number of nodes $N_1=N_2=N_3=2^7+1$ in each direction.
The stopping criterion is set to the default value.

% Tab. 3D, Ex2 (\epsilon_x = 1e-6)
\begin{table}[htdp]\caption{Comparison of numerical results computed by
Algorithm \ref{alg:ConOptM} (top half part with rows 2--4)
and BESP implemented in GPELab (bottom half part with rows 5--7)
for trapping potential (\ref{optic98}) with different $\beta$.
}%\footnotesize
\label{table-sp3d-ex2}
\setlength{\tabcolsep}{4pt}
\begin{center}
\begin{tabular}{ccccccccccc}\hline
$\beta$ & $\max|\phi_g|^2$  &  $E(\phi_g)$  &  $\mu_g$  &  $x_{\mathrm{rms}}$  &  $y_{\mathrm{rms}}$  &  $z_{\mathrm{rms}}$  &  iter & nfe & cpu\,(s) \\ \hline
100 & 0.2536 & 23.2356 & 27.4757  & 1.8716 & 1.8716 & 1.8716 & 112 & 115 & 76.47 \\% \hline
800 & 0.0490 & 33.8023 & 40.4476 & 2.6620 & 2.6620 & 2.6620 & 260 & 279 & 183.34 \\ %\hline
6400 & 0.0098 & 52.4955 & 63.7146 & 3.3685 & 3.3685 & 3.3685 & 305 & 327 & 217.03 \\[1mm] \hline\hline
100 & 0.2536 & 23.2356 & 27.4757  & 1.8717 & 1.8717 & 1.8717 & 188 & - & 914.53 \\% \hline
800 & 0.0490 & 33.8023 & 40.4476 & 2.6620 & 2.6620 & 2.6620 & 494 & - & 2513.75 \\ %\hline
6400 & 0.0098 & 52.4955 & 63.7149 & 3.3684 & 3.3684 & 3.3684 & 747 & - & 4014.17 \\ \hline
\end{tabular}
\end{center}
\end{table}

Table \ref{table-sp3d-ex2} depicts the maximum value of the wave
function $\max |{\phi_g}|^2$, the energy $E(\phi_g)$, the chemical potential
$\mu_g$ and the root mean squares $x_{\rm rms}$, $y_{\rm rms}$ and $z_{\rm rms}$
for different interaction constants $\beta$, as well as
the number of iterations (iter), the number of function evaluations (nfe)
and the computational time (cpu).
For comparison, we also display numerical results obtained
by using the BESP method implemented in GPELab~\cite{GPELab-2014}  (a MATLAB
toolbox designed for computing ground state and dynamics of BEC)
with time step taken as $\Delta t = 10^{-2}$ and all other setting the same as above.
In addition, Figure \ref{fig-sp3d-ex2} shows the isosurface plots and
their corresponding slice views of the ground states for different $\beta$.

% Fig. 3D, Ex2
\begin{figure}[htdp]\caption{Isosurfaces (left column) and their corresponding slice views (right column)
of the ground states $\phi_g(x,y,z)$ of a BEC in 3D with combined harmonic and optical
lattice potential (\ref{optic98}) for different $\beta=100,\,800,\,6400$ (from top to bottom).}
\label{fig-sp3d-ex2}
\begin{minipage}[t]{1\textwidth}
\centering
\subfigure[Isosurface plot]{
\begin{minipage}[t]{0.45\textwidth}
\centering
 \includegraphics[width=0.95\textwidth, height =
 0.8\textwidth]{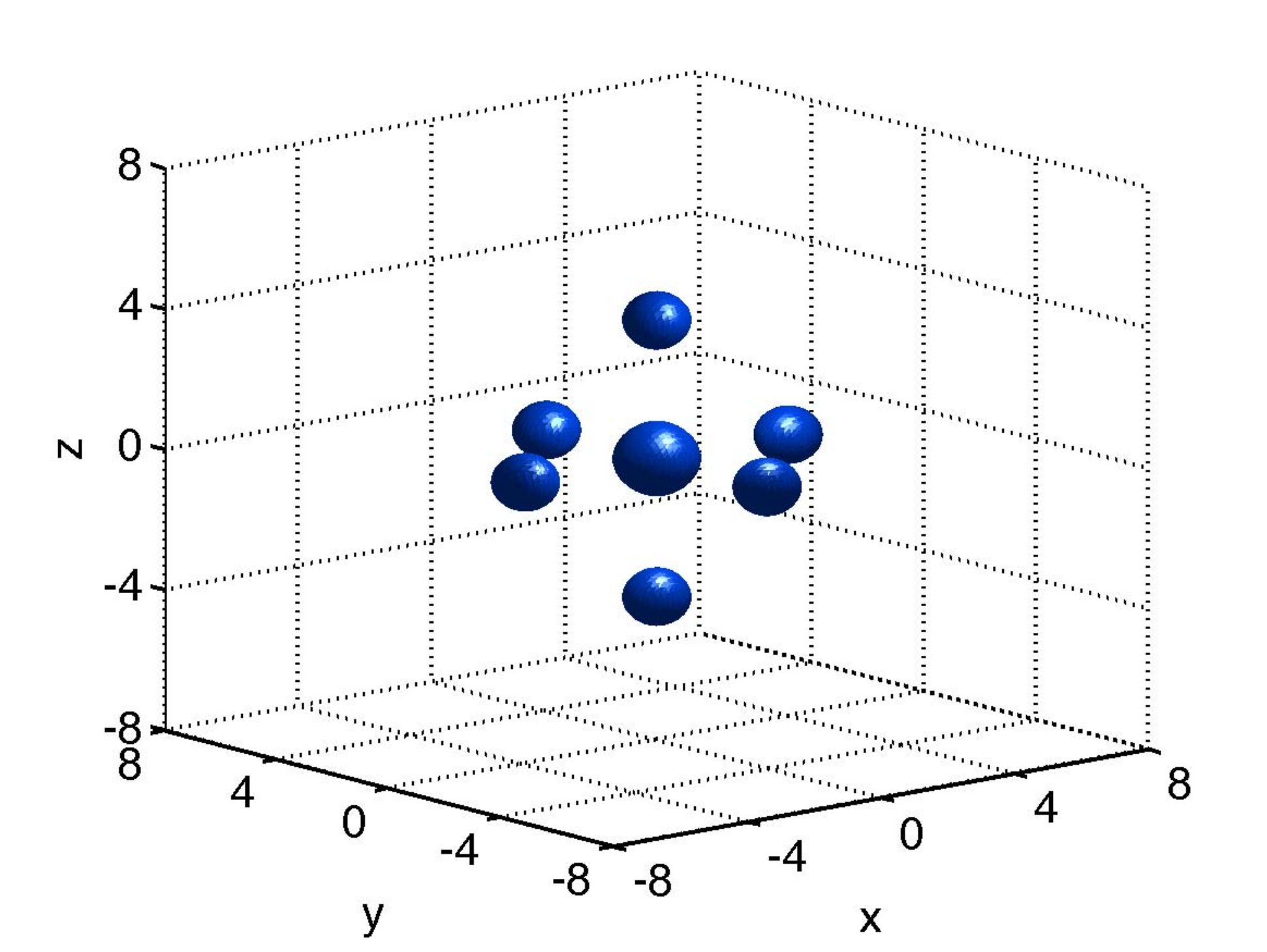}\\
\includegraphics[width=0.95\textwidth, height =
 0.8\textwidth]{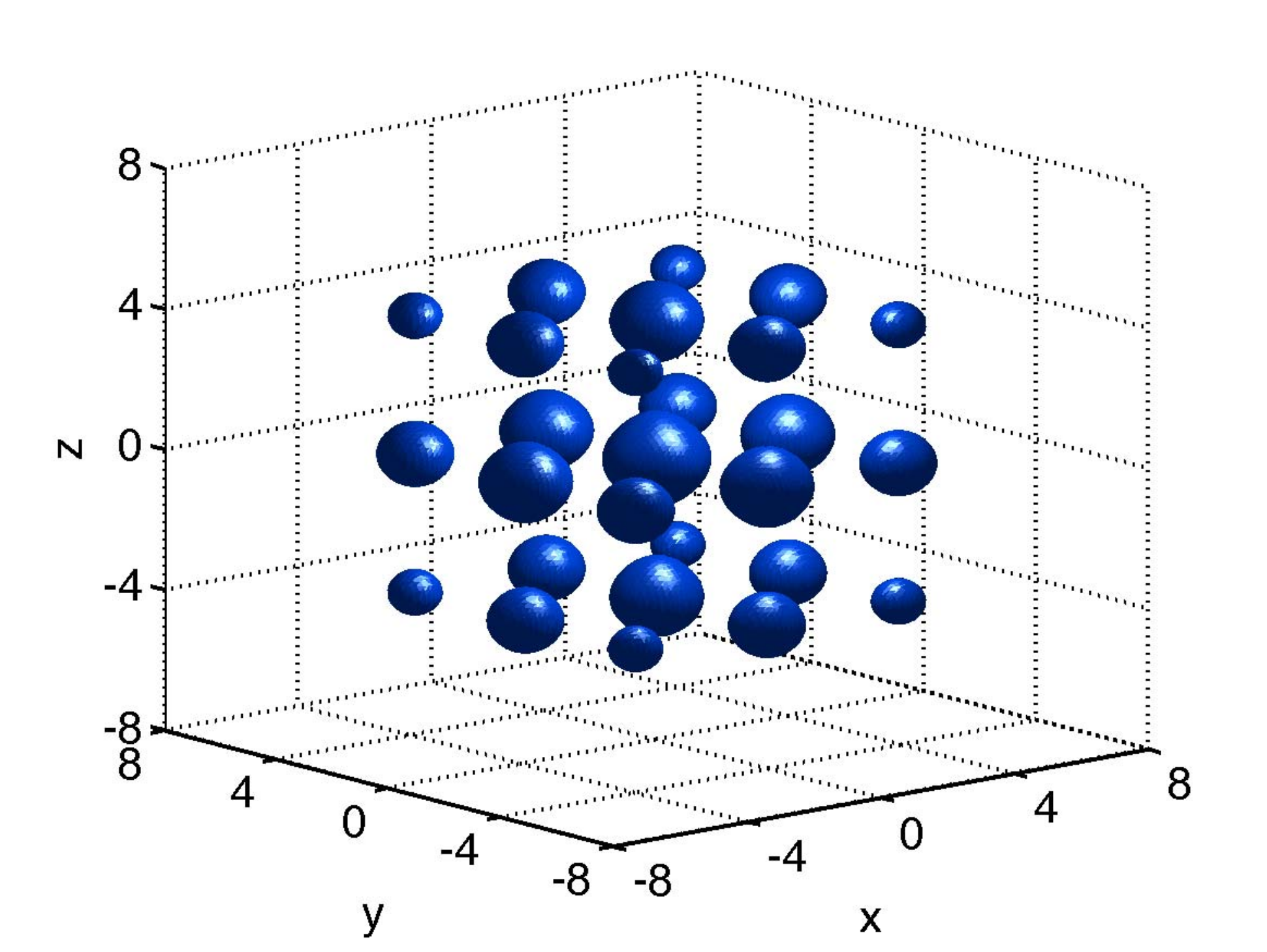}\\
 \includegraphics[width=0.95\textwidth, height =
 0.8\textwidth]{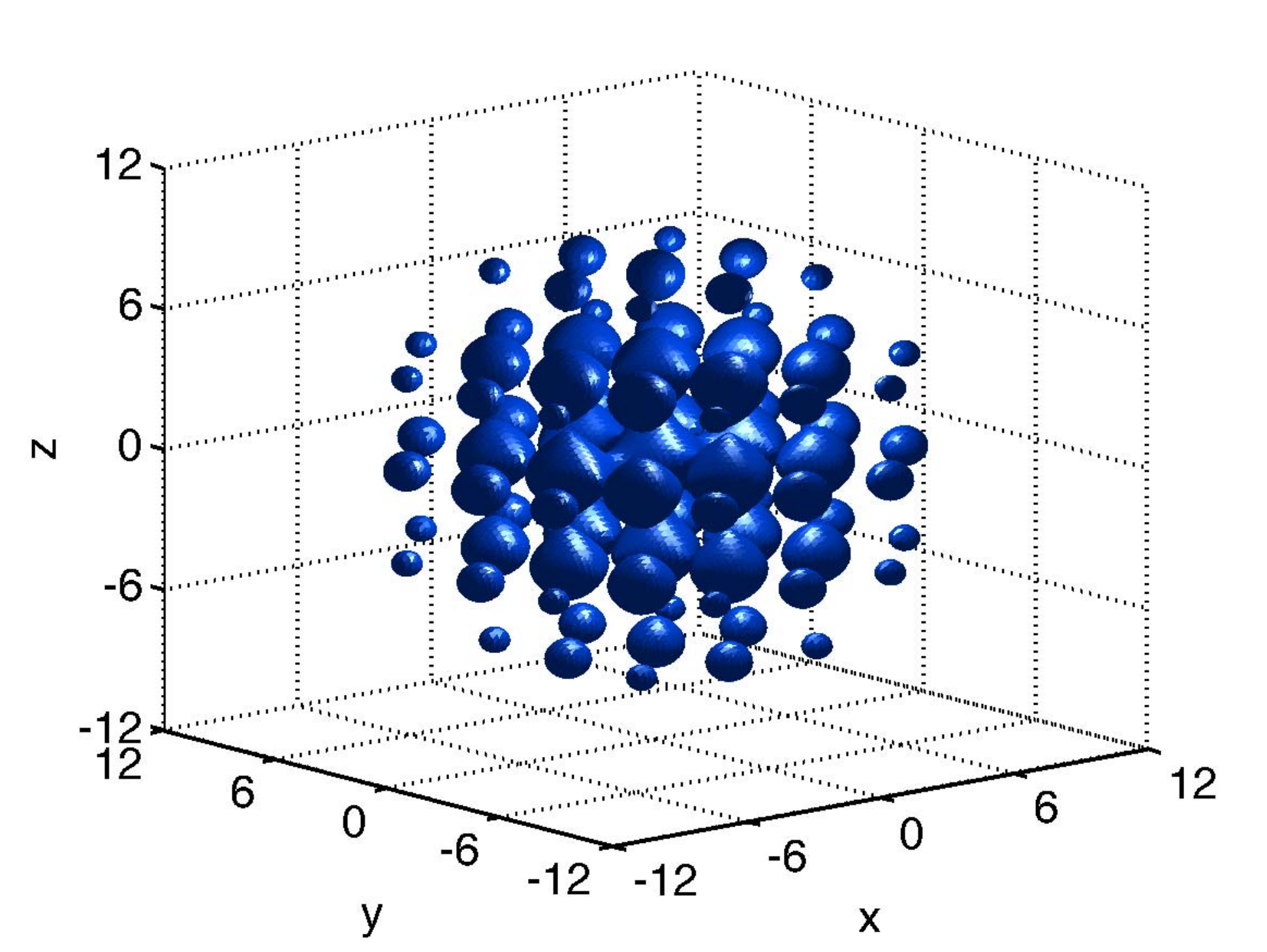}
\end{minipage}}
\subfigure[Slice plot]{
\begin{minipage}[t]{0.45\textwidth}
\centering
 \includegraphics[width=0.95\textwidth, height =
 0.8\textwidth]{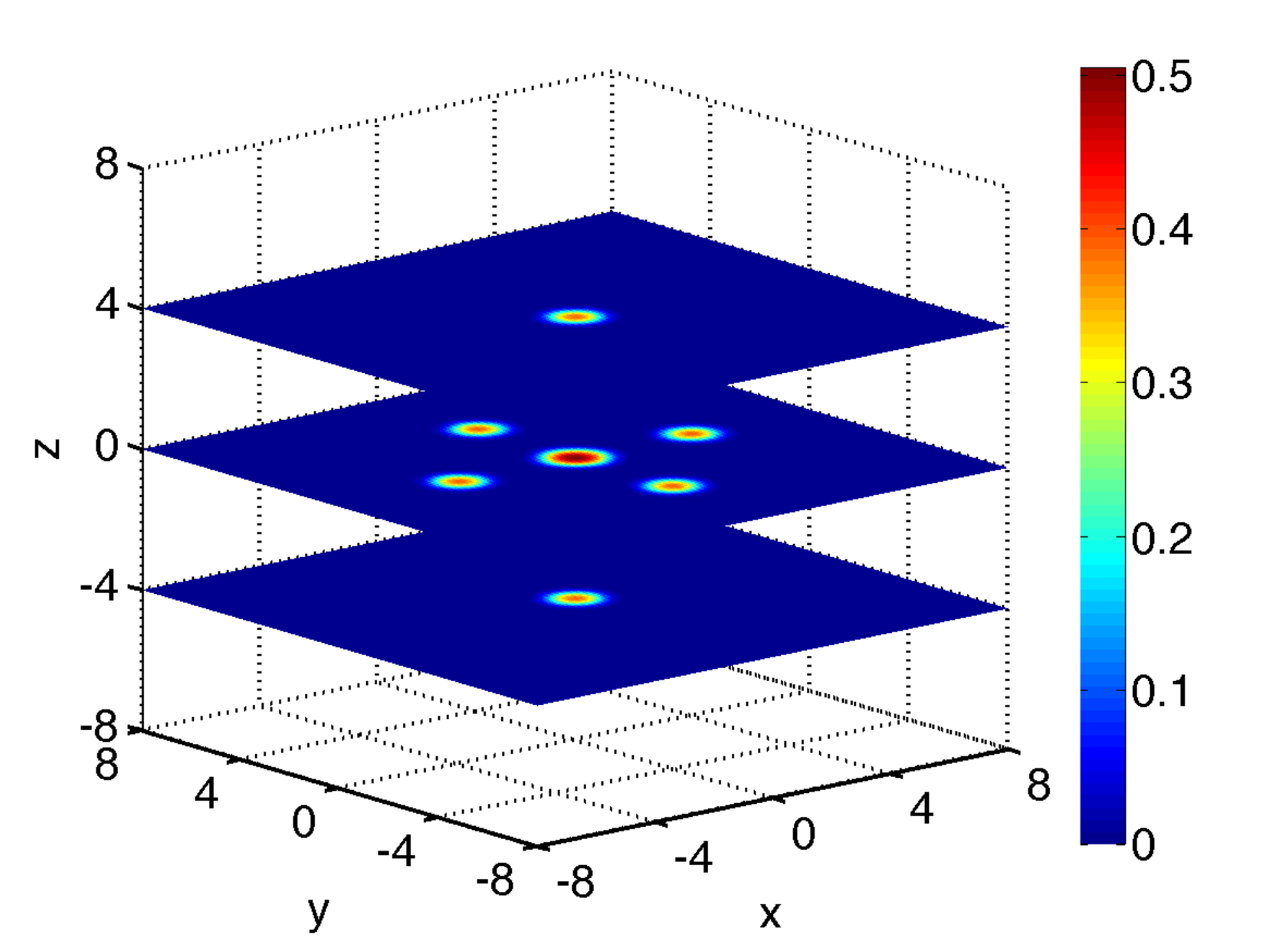}\\
% \includegraphics[width=0.95\textwidth, height =
% 0.8\textwidth]{figure/sp3d/results-3/SP3d-Bet-100-Init-Gaussian-xtol-1e-08-slice.eps}\\
 \includegraphics[width=0.95\textwidth, height =
 0.8\textwidth]{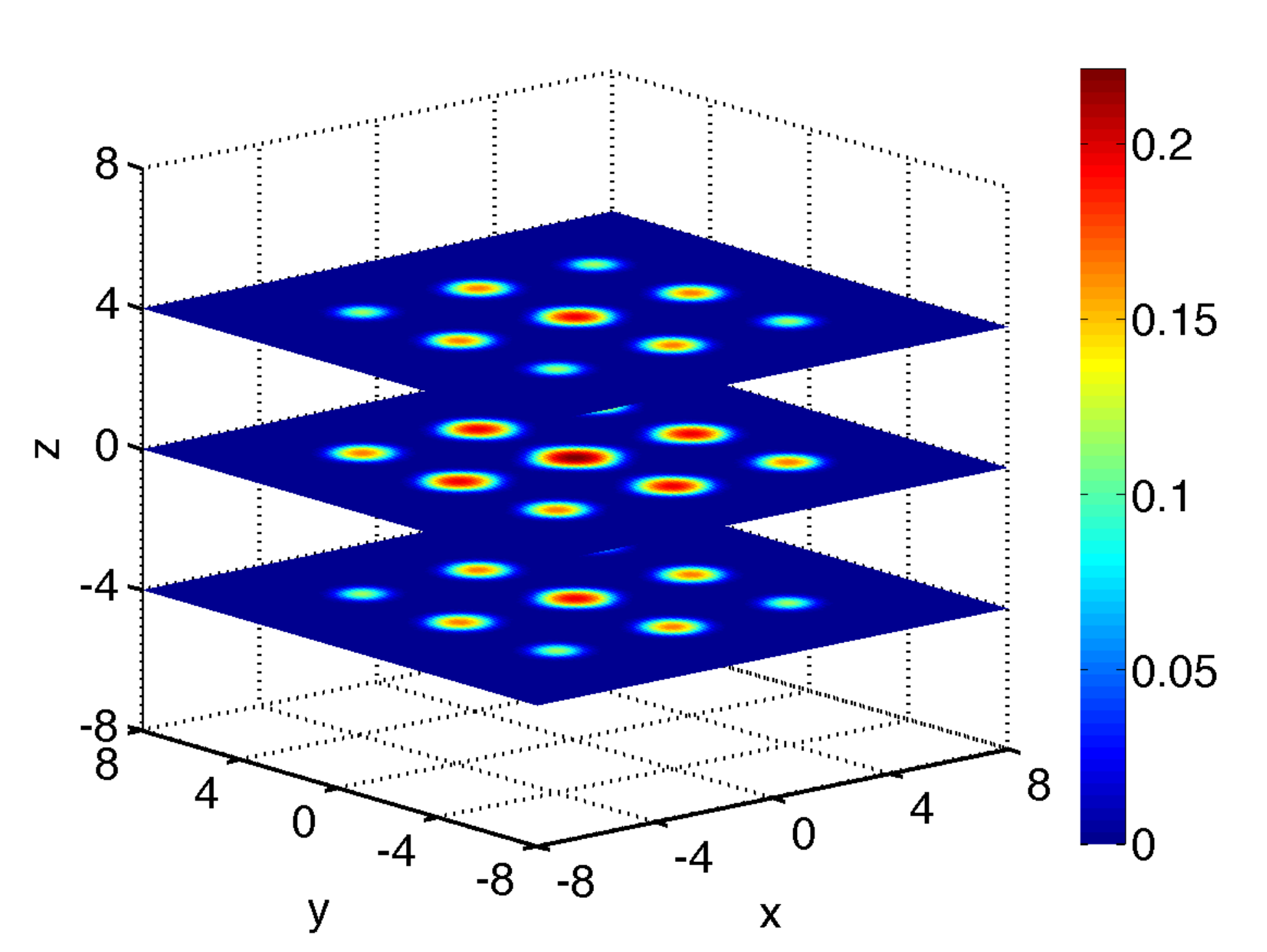}\\
 \includegraphics[width=0.95\textwidth, height =
 0.8\textwidth]{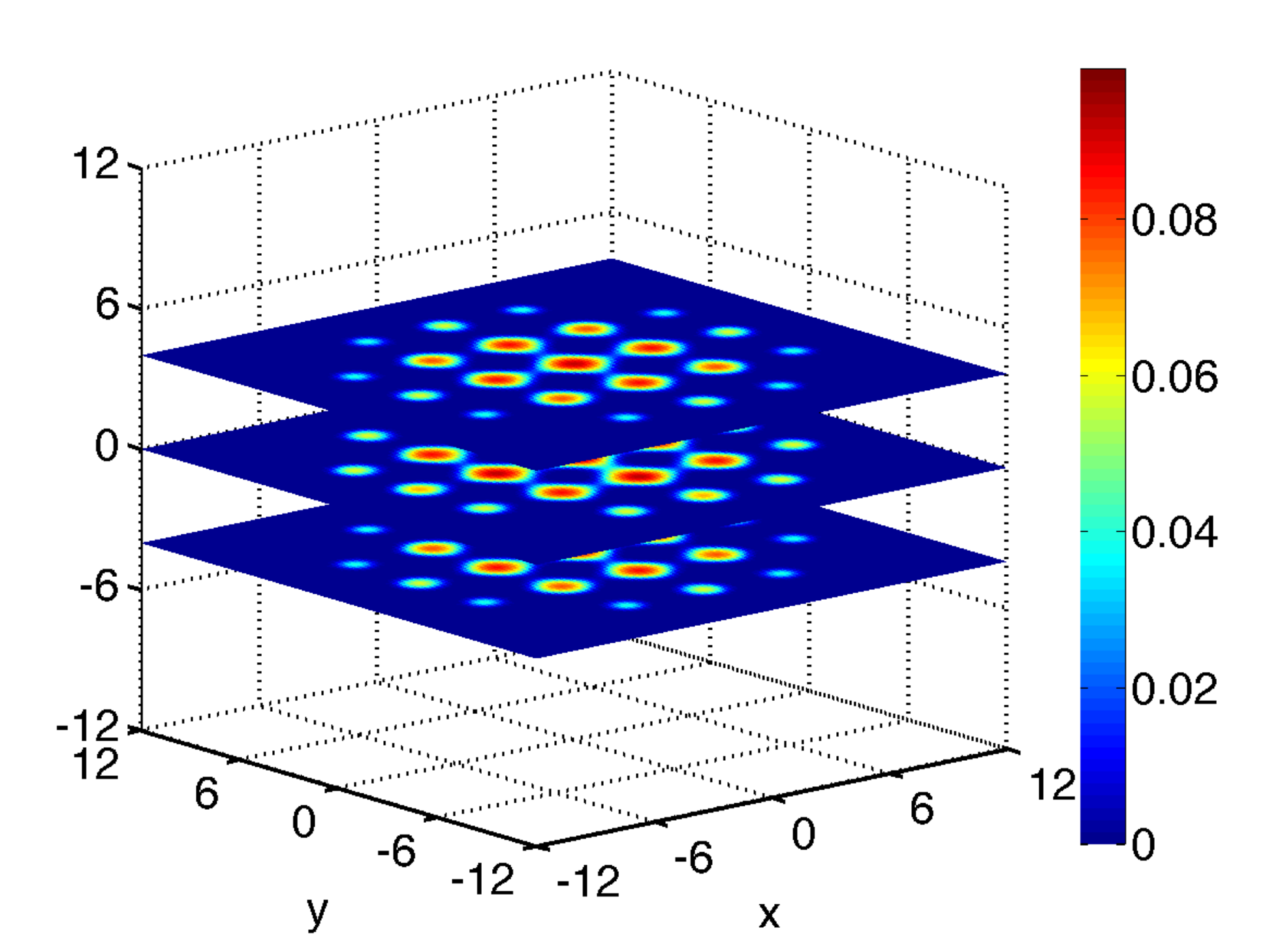}
 \end{minipage}}
\end{minipage}
\end{figure}

From Table \ref{table-sp3d-ex2}, we can see that
the Algorithm \ref{alg:ConOptM} converges to the ground state
much faster than the BESP method presented in
\cite{Bao-Cai-2013-review,Bao-Chern-Lim-2006} for all $\beta$
in computing the ground state of BEC in 3D.

%====================%====================
\subsection{Results for rotating BEC in 2D}
We take $d=2$ and the harmonic potential (\ref{harmpen}) with $\gamma_x=\gamma_y=1$
in (\ref{prob-min}) and (\ref{energy}) and consider different $\beta$ and $\Omega$.
The ground state is numerically computed by  the regularized Newton method (i.e. Algorithm \ref{alg:TR})
with the FP discretization
on bounded computational domains
$U=(-10,10)^2$ and $U=(-12,12)^2$ for $\beta=500$ and $\beta=1000$,
respectively. The domains are partitioned uniformly with the number of nodes
$N_1=N_2=2^8+1$ in each direction.
 In our computations, in the Algorithm \ref{alg:TR}, we first call the gradient type method,
i.e., Algorithm \ref{alg:ConOptM}, with
a maximum number of iterations $K_{\rm init}=100$ to obtain a good initial guess $X^{(1)}$.
Then the  regularized Newton subproblem is  solved by the Algorithm
\ref{alg:ConOptM} up to a maximum number of iterations $K_{\rm sub}=200$.
In order to reduce the computational cost,
the cascadic multigrid method (i.e., Algorithm \ref{alg:meshrefine}) is applied
for mesh refinement with the coarsest mesh $\mathcal{T}^0$ chosen with the number of nodes $N_1=N_2=2^4+1$
in each direction.

For a rotating BEC, the ground state is a complex-valued function, and thus it is very tricky
to choose a proper initial data such that the numerical result is guaranteed to be the ground state.
Similarly to those in the literatures \cite{Bao-Wang-Markowich-2005}, here we test our
algorithms with the following different initial solutions
\beaa
&& {\rm (a)} \ \phi_{a}(x,y) =  \frac{1}{\sqrt{\pi}}e^{-(x^2+y^2)/2},\\
&& {\rm (b)} \ \phi_{b}(x,y) = \frac{x+iy}{\sqrt{\pi}}e^{-(x^2+y^2)/2}, \qquad (\bar{b})\ \phi_{\bar{b}}(x,y)=\bar{\phi}_{b}(x,y),\\
&& {\rm (c)} \ \phi_{c}(x,y) = \frac{[\phi_a(x,y))+\phi_b(x,y)]/2}{\|[\phi_a(x,y))+\phi_b(x,y)]/2\|},
\qquad (\bar{c})\ \phi_{\bar{c}}(x,y)=\bar{\phi}_{c}(x,y),\\
&& {\rm (d)} \ \phi_{d}(x,y) = \frac{(1-\Omega)\phi_a(x,y))+\Omega\phi_b(x,y)}{\|(1-\Omega)\phi_a(x,y))+\Omega\phi_b(x,y)\|},
\qquad (\bar{d})\ \phi_{\bar{b}}(x,y)=\bar{\phi}_{d}(x,y).
\eeaa

Table \ref{tab-fp2d-Rot-Bet-500-energy} displays the energy obtained numerically
with different initial data selected in the above with $\beta=500$ for different
$\Omega=0.00$, $0.25$, $0.50$, $0.60$, $0.70$, $0.80$, $0.90$ and $0.95$
(in the table, we use a ``$\dagger$'' sign to indicate the one with the lowest energy among
different initial data for given $\beta$ and $\Omega$),
and Table \ref{tab-fp2d-Rot-Bet-500-iter-cpu} summarizes the lowest energy among different initial
data and the corresponding
number of iterations and computation time for $\beta=500$ with different $\Omega$.
Figure \ref{fig-fp2d-Rot-Bet-500} plots the ground state density $|\phi_g(x,y)|^2$
for $\beta=500$ with different $\Omega$.
In addition, Tables \ref{tab-fp2d-Rot-Bet-1000-energy}-\ref{tab-fp2d-Rot-Bet-1000-iter-cpu}
 and Figure \ref{fig-fp2d-Rot-Bet-1000} present similar numerical results for $\beta=1000$.

% ---- \beta=500 (Tab1)----
\begin{table}[htdp]\caption{Energy obtained numerically with different initial data
of rotating BECs for $\beta=500$  and different $\Omega$ in \S 4.3.}%\footnotesize
\label{tab-fp2d-Rot-Bet-500-energy}
\setlength{\tabcolsep}{6pt}
\begin{center}
%\begin{tabular}{ccccccccc}\hline
\begin{tabular}{lllllllll}\hline
$\Omega$ & 0.00 & 0.25 & 0.50 & 0.60 & 0.70 & 0.80 & 0.90 & 0.95 \\ \hline
$({\rm a})$   & 8.5118  & 8.5118  & 8.0246  & 7.5890  & 6.9731  & 6.1016  & 4.7778  & 3.7417  \\ %\hline
$({\rm b})$  & 8.5118  & 8.5106  & 8.0246  & 7.5845  & 6.9731  & 6.1055  & 4.7778  & 3.7417  \\ %\hline
$({\rm\bar b})$  & 8.5118  & 8.5118  & 8.0197$^\dag$  & 7.5890  & 6.9731  & 6.1016  & 4.7778  & 3.7416  \\ %\hline
$({\rm c})$ & 8.5118  & 8.5106  & 8.0246  & 7.5890  & 6.9726  & 6.1016  & 4.7778 & 3.7417  \\ %\hline
$({\rm\bar c})$  & 8.5118  & 8.5118  & 8.0246  & 7.5890  & 6.9731  & 6.0997  & 4.7778  & 3.7415  \\ %\hline
$({\rm d})$  & 8.5118$^\dag$  & 8.5106$^\dag$  & 8.0246  & 7.5890  & 6.9726$^\dag$ & 6.0997$^\dag$  & 4.7778$^\dag$  & 3.7415$^\dag$  \\ %\hline
$({\rm\bar d})$  & 8.5118  & 8.5118  & 8.0246  & 7.5845$^\dag$  & 6.9731  & 6.1016  & 4.7778 & 3.7416 \\ \hline
\end{tabular}
\end{center}
\end{table}

% ---- \beta=500 (Tab2)----
\begin{table}[htdp]\caption{Ground state energy, the number of iterations for
the regularized Newton method (iter) on the finest mesh and the total computational time (cpu)
of rotating BECs for $\beta=500$  and different $\Omega$ in \S 4.3.}%\footnotesize
\label{tab-fp2d-Rot-Bet-500-iter-cpu}
\setlength{\tabcolsep}{6pt}
\begin{center}
\begin{tabular}{ccccccccc}\hline
$\Omega$ & 0.00 & 0.25 & 0.50 & 0.60 & 0.70 & 0.80 & 0.90 & 0.95 \\ \hline
{\rm iter}    &   3  &   3  &   3 &  128 &  49 &  18 &  69 & 4  \\ %\hline
{\rm cpu\,(s)}   & 1.14 & 18.71 & 41.57 &  355.63 &  147.03 &  130.87 &  286.12 & 56.08  \\ \hline \hline
{\rm Energy} & 8.5118  & 8.5106  & 8.0197  & 7.5845  & 6.9726  & 6.0997  & 4.7778 & 3.7415 \\ \hline
\end{tabular}
\end{center}
\end{table}

% ---- \beta=500 (Fig)----
\begin{figure}[htdp]\caption{Plots of the ground state density
  $|\phi_g(x,y)|^2$ -- corresponding to the energy listed in the Table \ref{tab-fp2d-Rot-Bet-500-energy} --
of rotating BECs for $\beta=500$ and different $\Omega$ in \S4.3.}
\label{fig-fp2d-Rot-Bet-500}
\begin{minipage}[t]{0.94\textwidth}
\centering
 \includegraphics[width=0.4\textwidth]{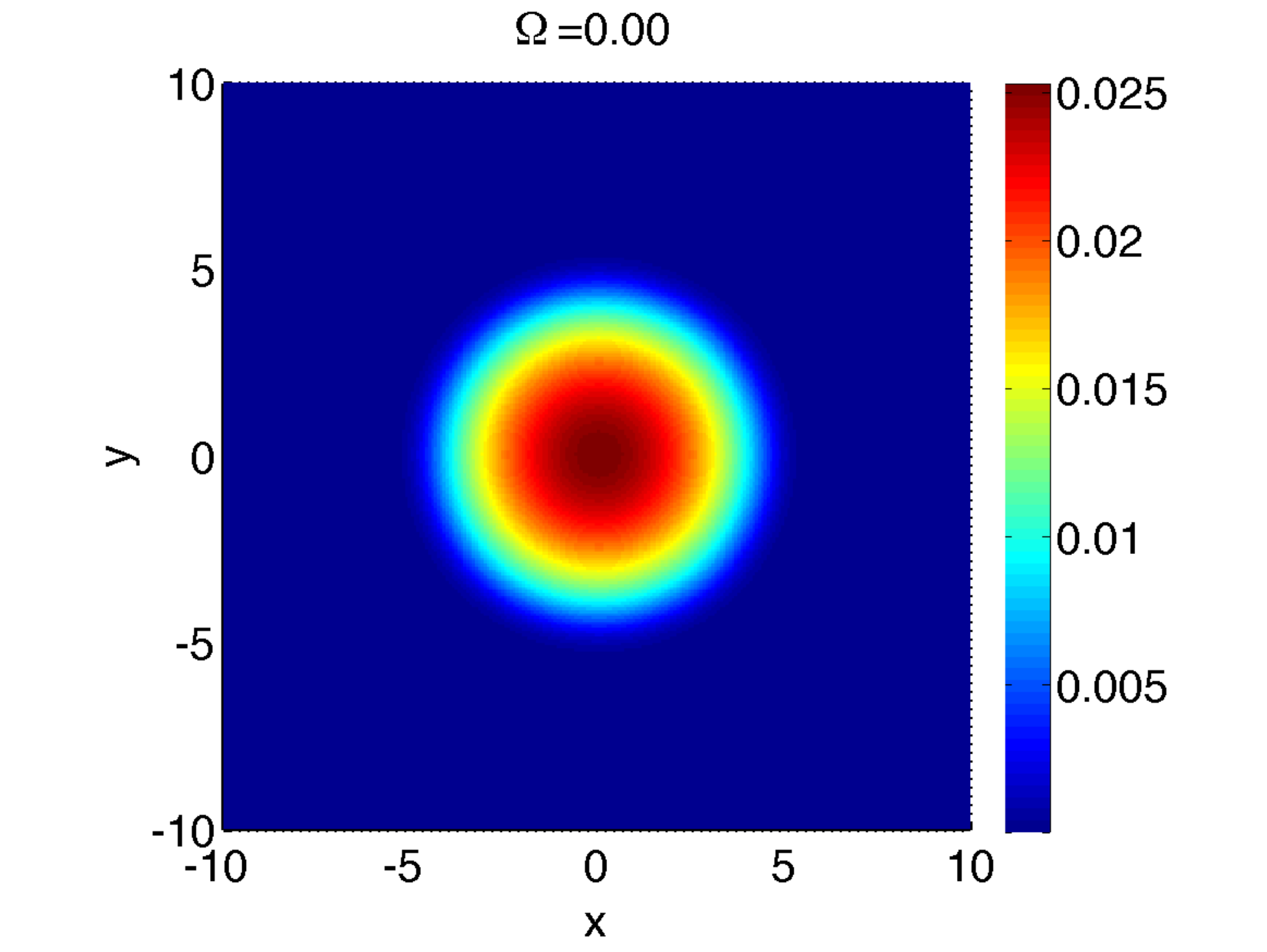}
 \includegraphics[width=0.4\textwidth]{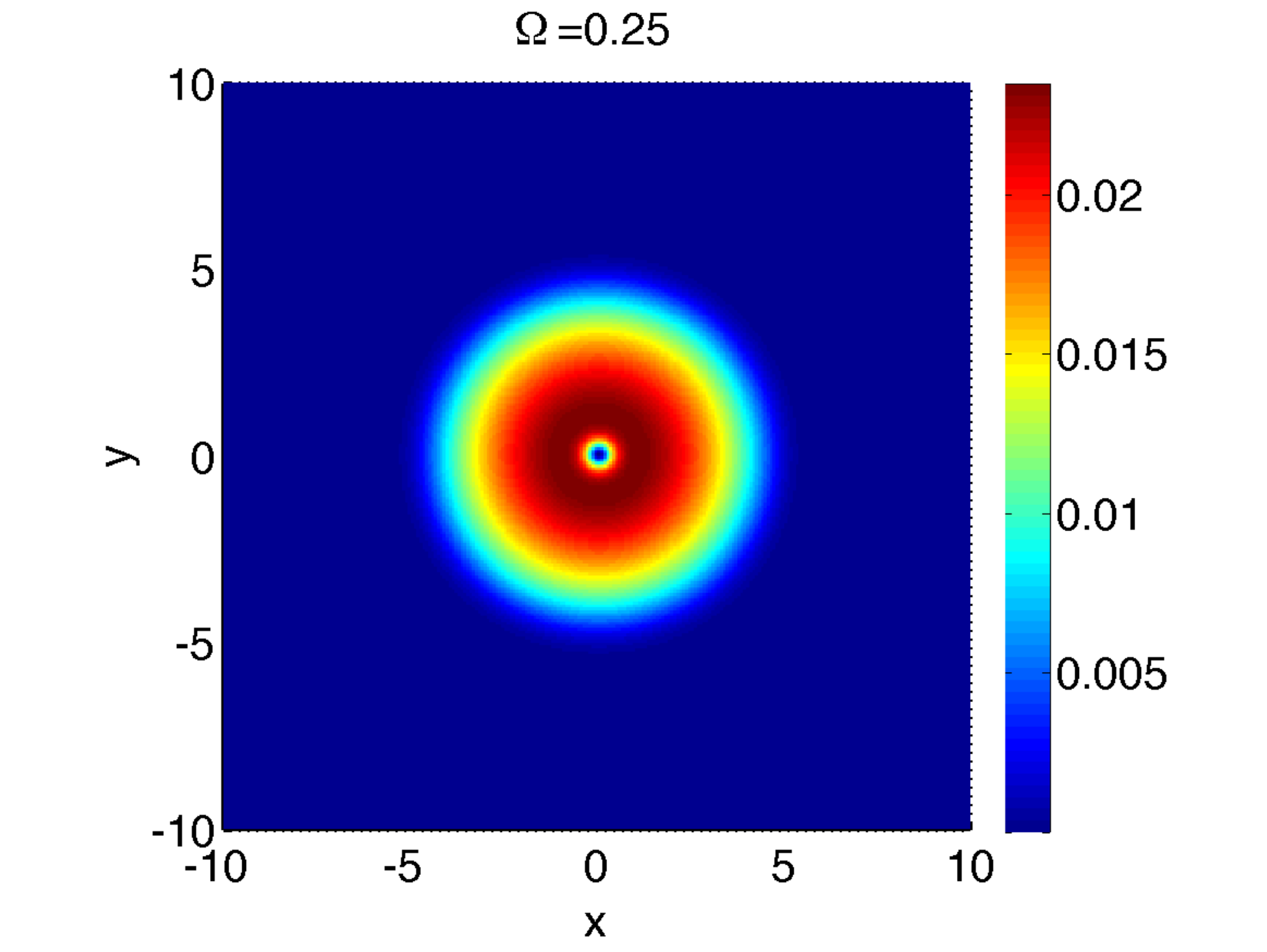}\\
 \includegraphics[width=0.4\textwidth]{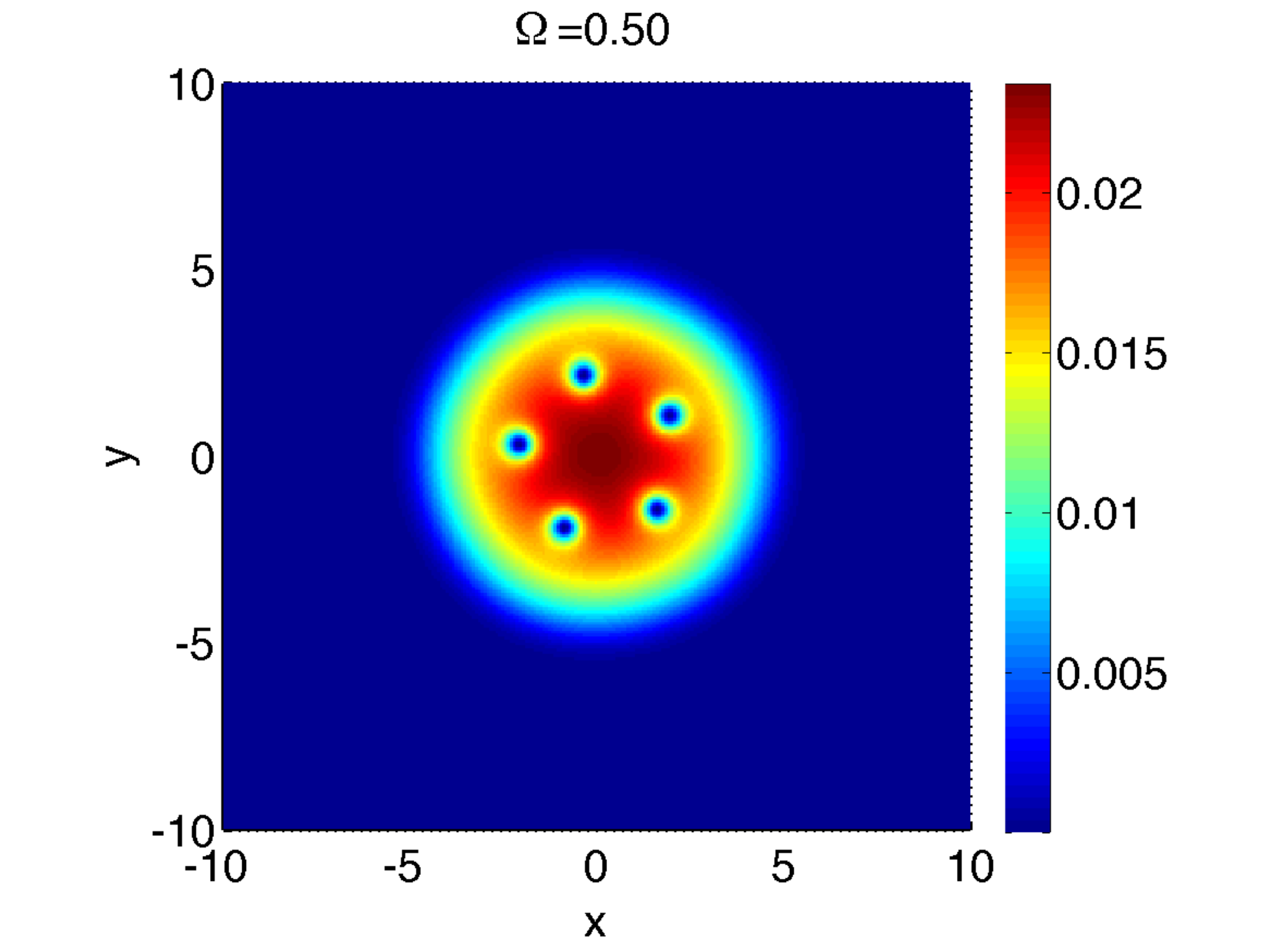}
 \includegraphics[width=0.4\textwidth]{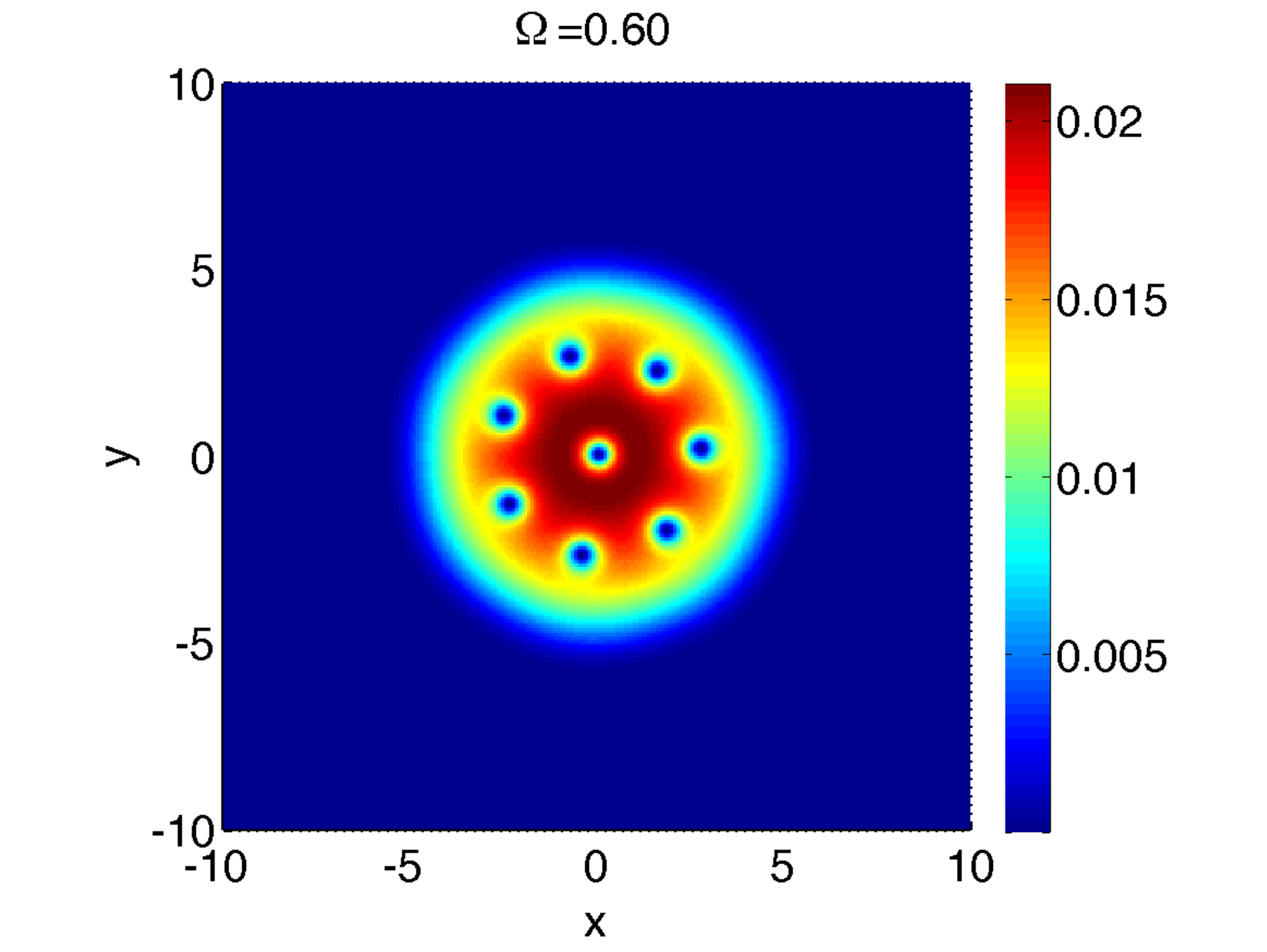}\\
 \includegraphics[width=0.4\textwidth]{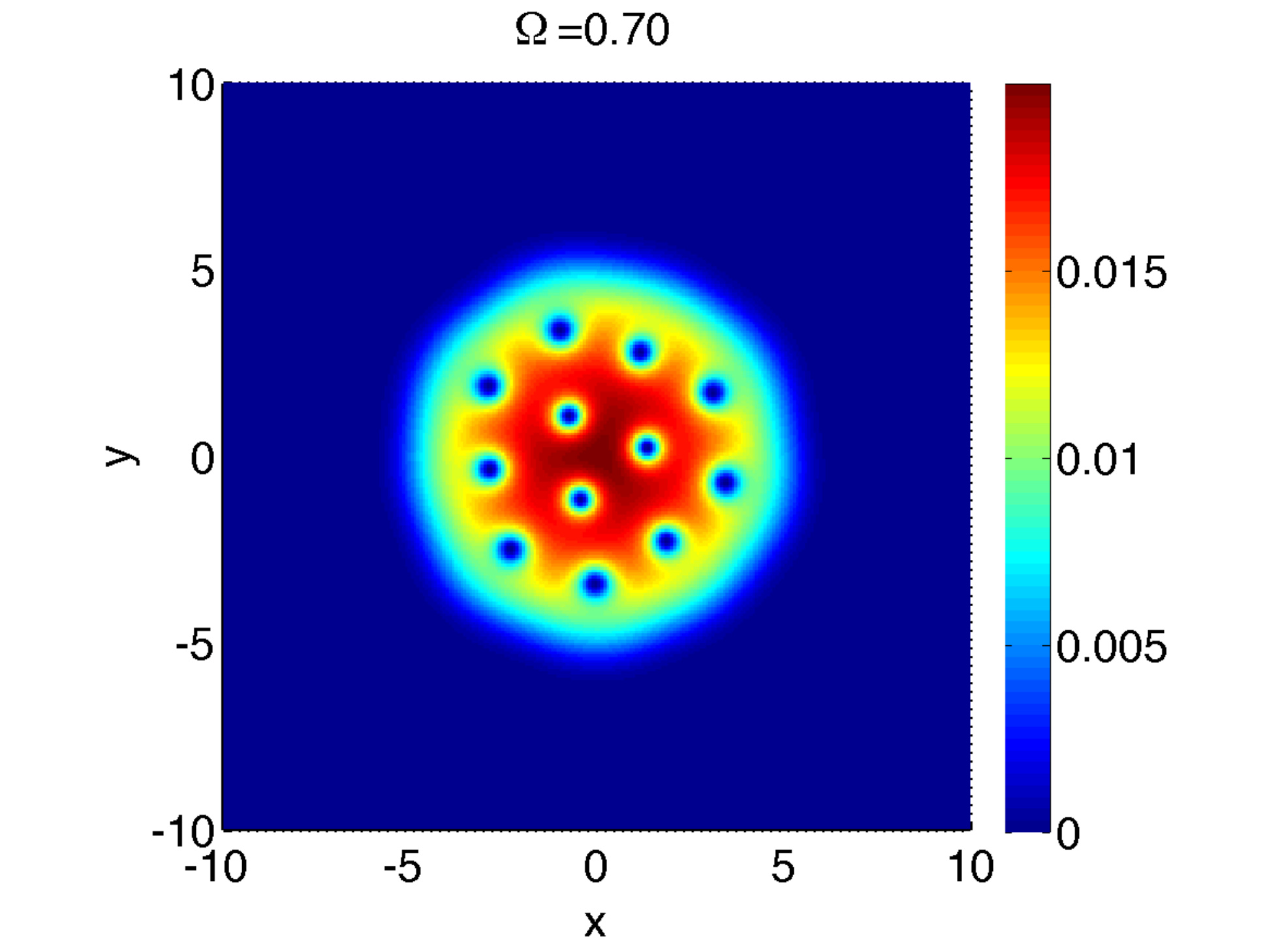}
 \includegraphics[width=0.4\textwidth]{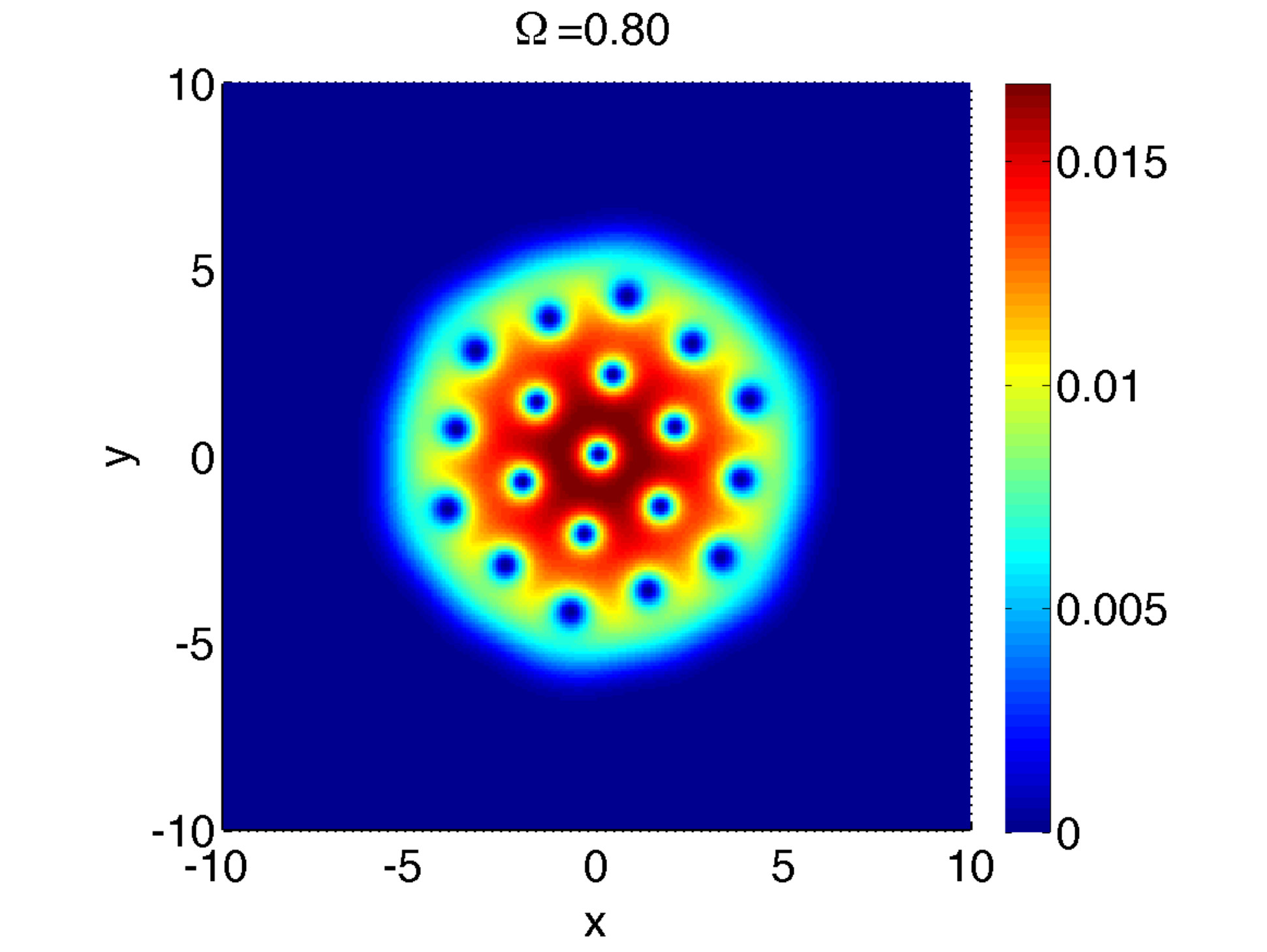}\\
 \includegraphics[width=0.4\textwidth]{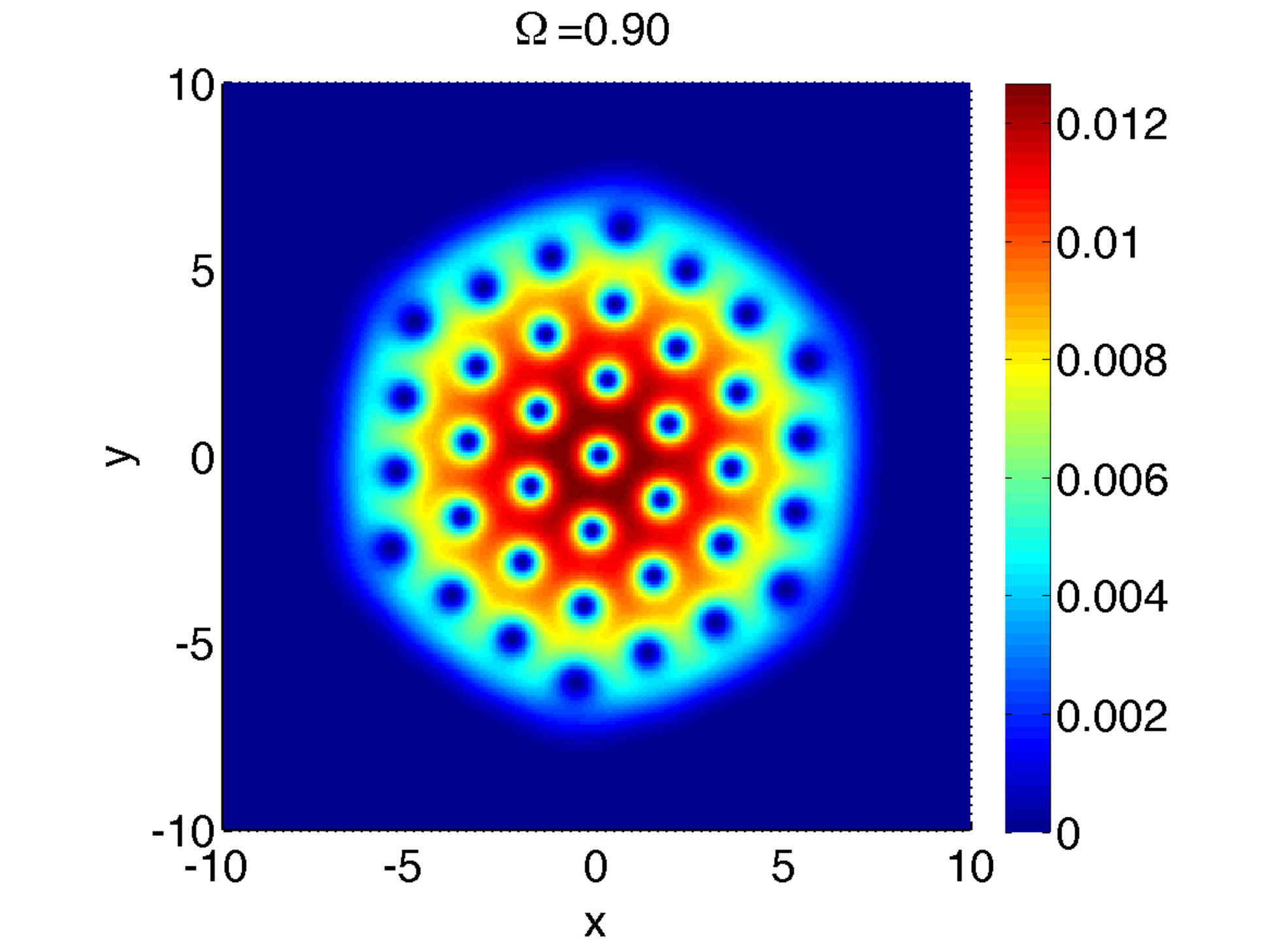}
 \includegraphics[width=0.4\textwidth]{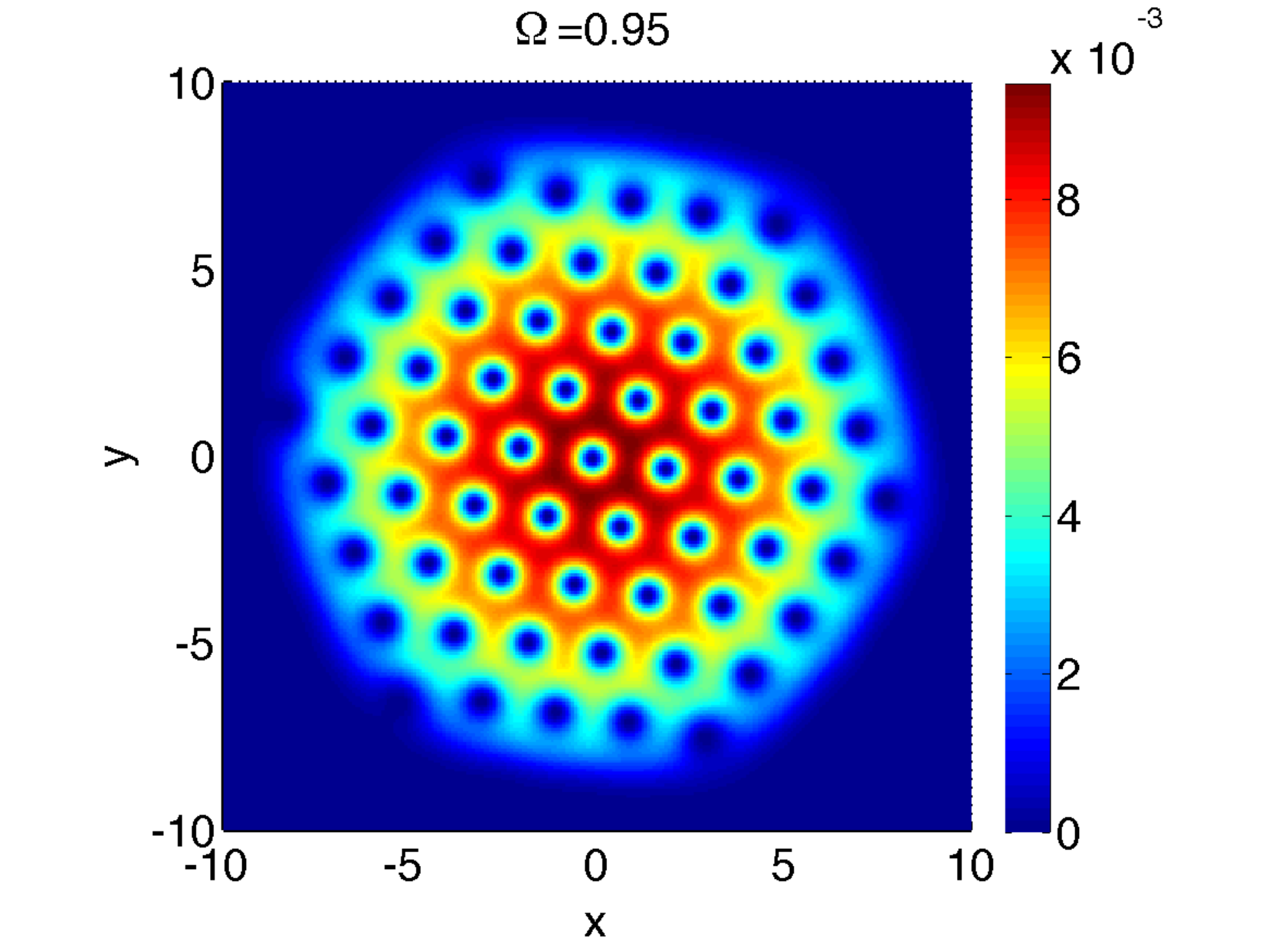}
 \end{minipage}
\end{figure}

% ---- \beta=1000 (Tab1)----
\begin{table}[htdp]\caption{
Energy obtained numerically with different initial data
of rotating BECs for $\beta=1000$  and different $\Omega$ in \S 4.3.}%\footnotesize
\label{tab-fp2d-Rot-Bet-1000-energy}
\setlength{\tabcolsep}{4pt}
\begin{center}
%\begin{tabular}{ccccccccc}\hline
\begin{tabular}{lllllllll}\hline
$\Omega$ & 0.00 & 0.25 & 0.50 & 0.60 & 0.70 & 0.80 & 0.90 & 0.95 \\ \hline
$({\rm a})$  & 11.9718  & 11.9718  & 11.0954$^\dag$  & 10.4392 & 9.5335  & 8.2610  & 6.3608  & 4.8830  \\ %\hline
$({\rm b})$  & 11.9718  & 11.9266  & 11.1326  & 10.4392  & 9.5283  & 8.2610  & 6.3607  & 4.8825  \\ %\hline
$({\rm\bar b})$  & 11.9718  & 11.9266  & 11.1054  & 10.4392  & 9.5335 & 8.2631  & 6.3607  & 4.8827  \\ %\hline
$({\rm c})$  & 11.9718  & 11.9165  & 11.1054  & 10.4392  & 9.5289
& 8.2610  & 6.3607  & 4.8823$^\dag$ \\ %\hline
$({\rm\bar c})$  & 11.9718  & 11.9165  & 11.1326  & 10.4392  & 9.5283  & 8.2610  & 6.3607 & 4.8825 \\ %\hline
$({\rm d})$  & 11.9718  & 11.9266  & 11.1054  & 10.4392  & 9.5289 & 8.2632  & 6.3608  & 4.8825  \\ %\hline
$({\rm\bar d})$  & 11.9718$^\dag$  & 11.9165$^\dag$  & 11.1326  & 10.4392$^\dag$ & 9.5283$^\dag$  & 8.2610$^\dag$  & 6.3607$^\dag$  & 4.8830  \\ \hline
\end{tabular}
\end{center}
\end{table}

% ---- \beta=1000 (Tab2)----
\begin{table}[htdp]\caption{Ground state energy, the number of iterations for
the regularized Newton method (iter) on the finest mesh and the total computational time (cpu)
of rotating BECs for $\beta=1000$  and different $\Omega$ in \S 4.3.}%\footnotesize
\label{tab-fp2d-Rot-Bet-1000-iter-cpu}
\setlength{\tabcolsep}{4pt}
\begin{center}
%\begin{tabular}{|c|c|c|c|c|c|c|c|c|}\hline
\begin{tabular}{ccccccccc}\hline
$\Omega$ & 0.00 & 0.25 & 0.50 & 0.60 & 0.70 & 0.80 & 0.90 & 0.95 \\ \hline
{\rm iter}    &   3  &   3  &    3 &  10 &  10 &  72 &  41 & 157  \\ %\hline
{\rm cpu\,(s)}   & 1.18 &  28.52 & 108.98 &  106.86 &  105.28 &  313.67 &  825.12 & 751.72  \\ \hline \hline
{\rm Energy} & 11.9718  & 11.9165  & 11.0954  & 10.4392 & 9.5283  & 8.2610  & 6.3607  & 4.8823\\ \hline
\end{tabular}
\end{center}
\end{table}

% ---- \beta=1000 (Fig)----
\begin{figure}[htdp]\caption{Plots of the ground state density
  $|\phi_g(x,y)|^2$ -- corresponding to the energy listed in the Table \ref{tab-fp2d-Rot-Bet-1000-energy} --
of rotating BECs for $\beta=1000$ and different $\Omega$ in \S4.3.}
\label{fig-fp2d-Rot-Bet-1000}
\begin{minipage}[t]{\textwidth}
\centering
 \includegraphics[width=0.4\textwidth]{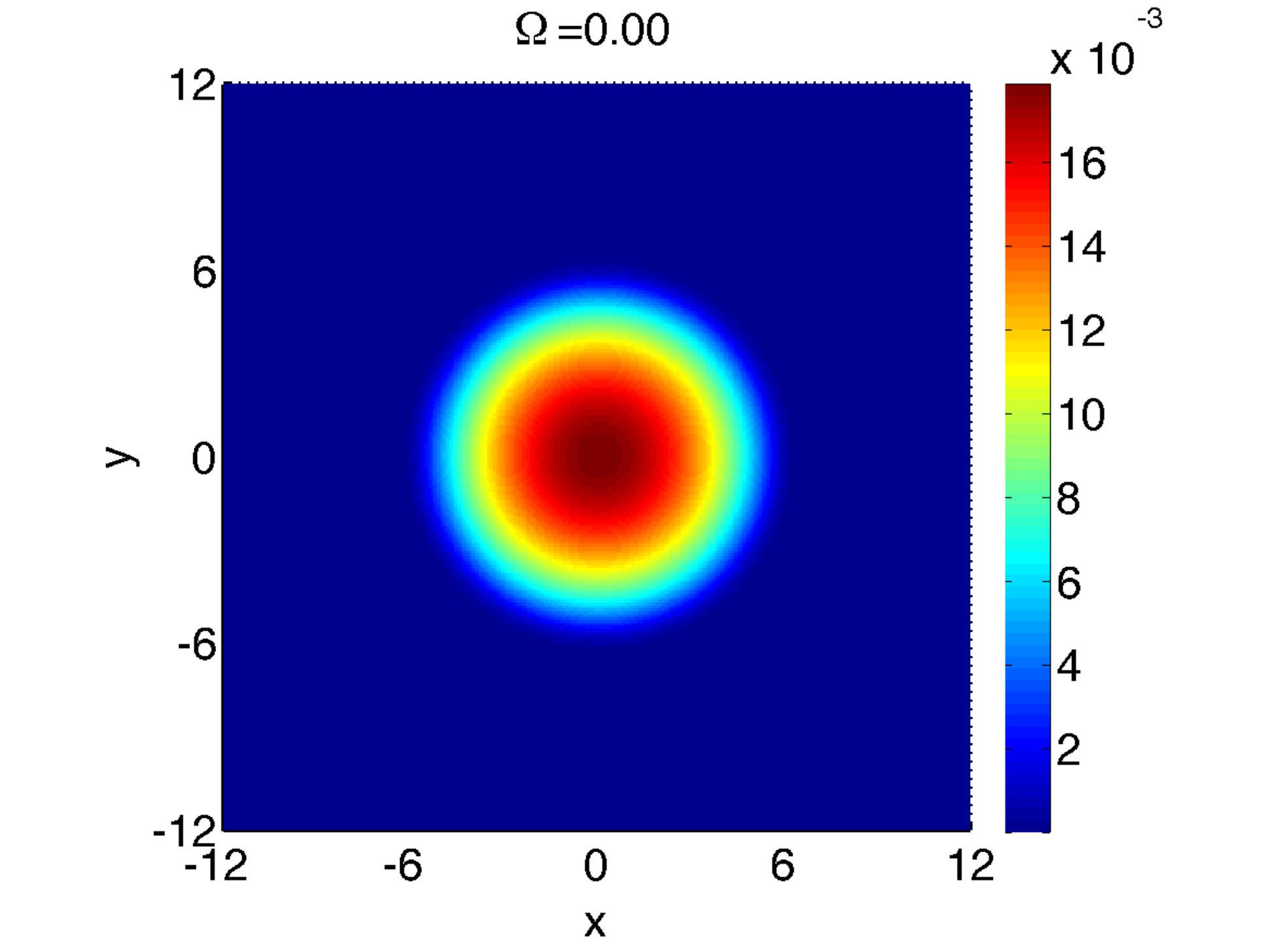}
 \includegraphics[width=0.4\textwidth]{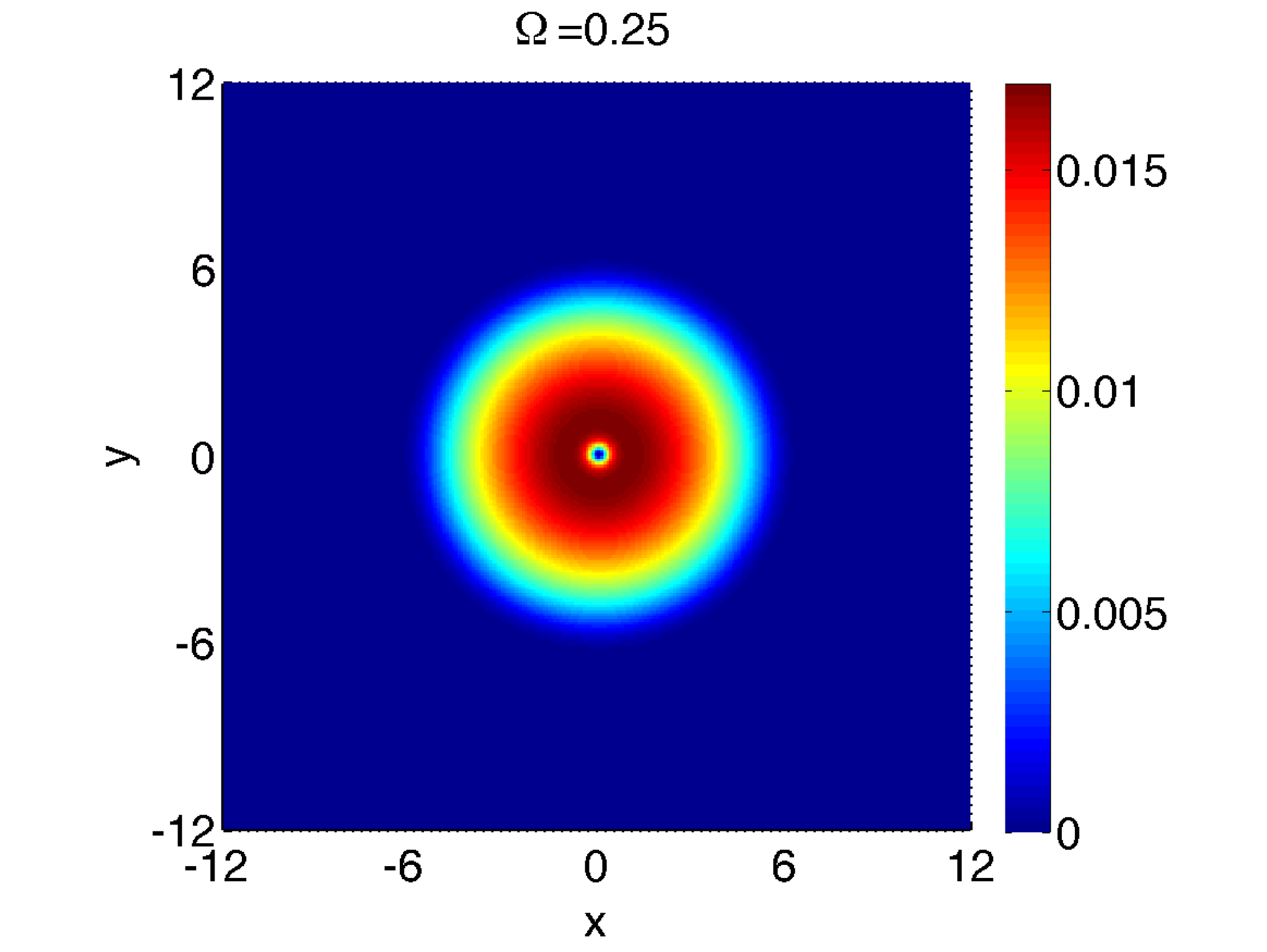}\\
 \includegraphics[width=0.4\textwidth]{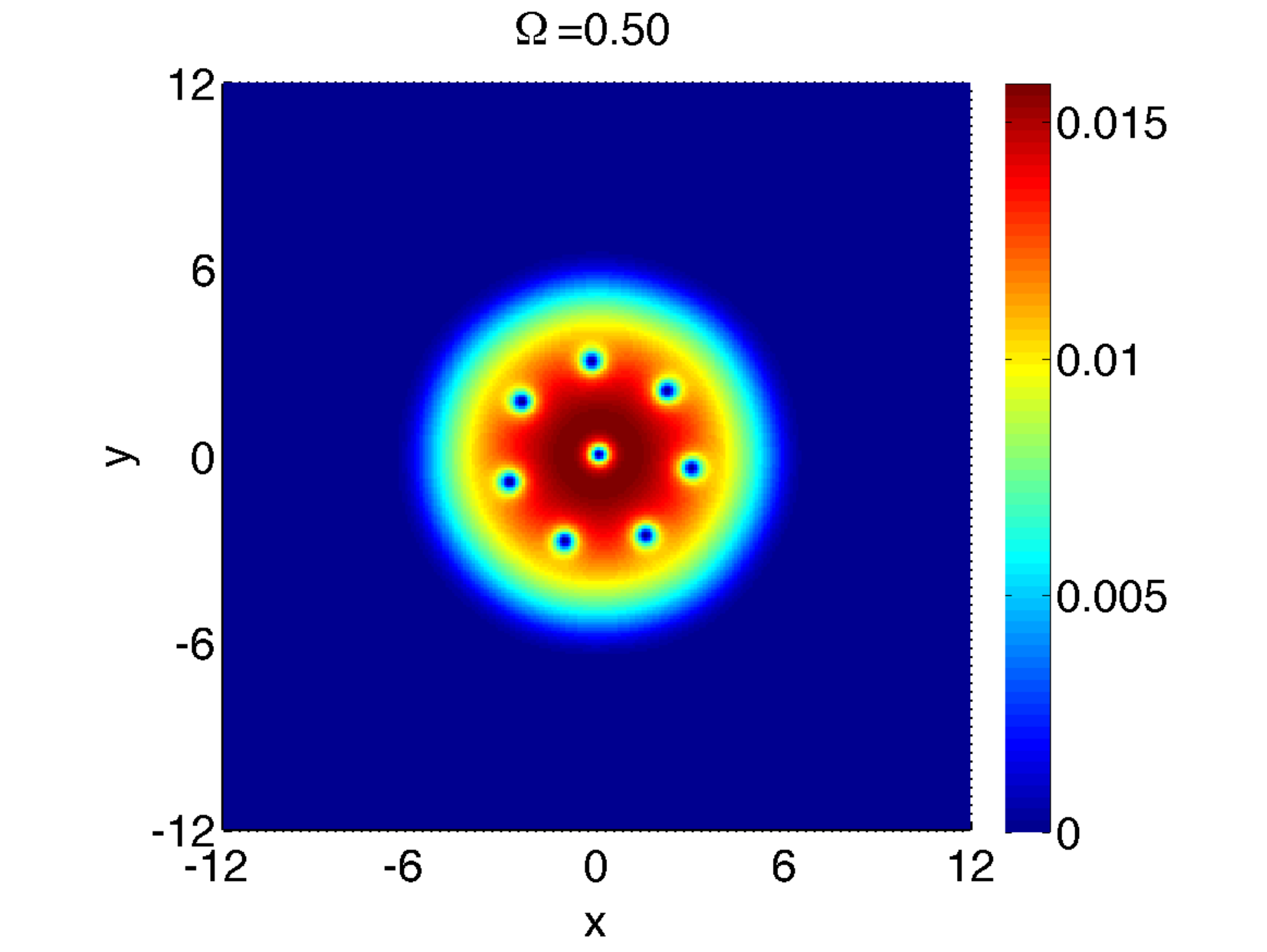}
 \includegraphics[width=0.4\textwidth]{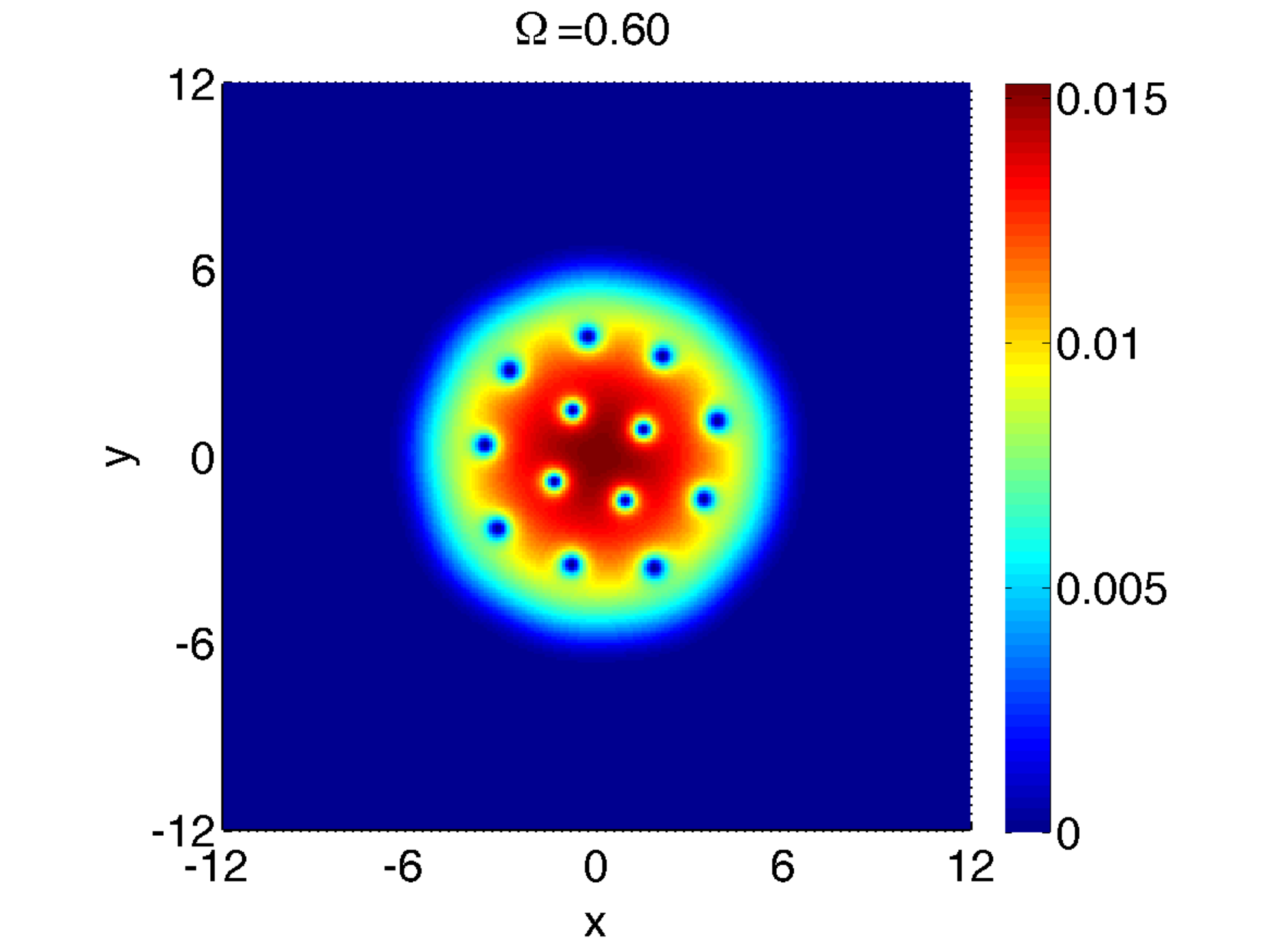}\\
 \includegraphics[width=0.4\textwidth]{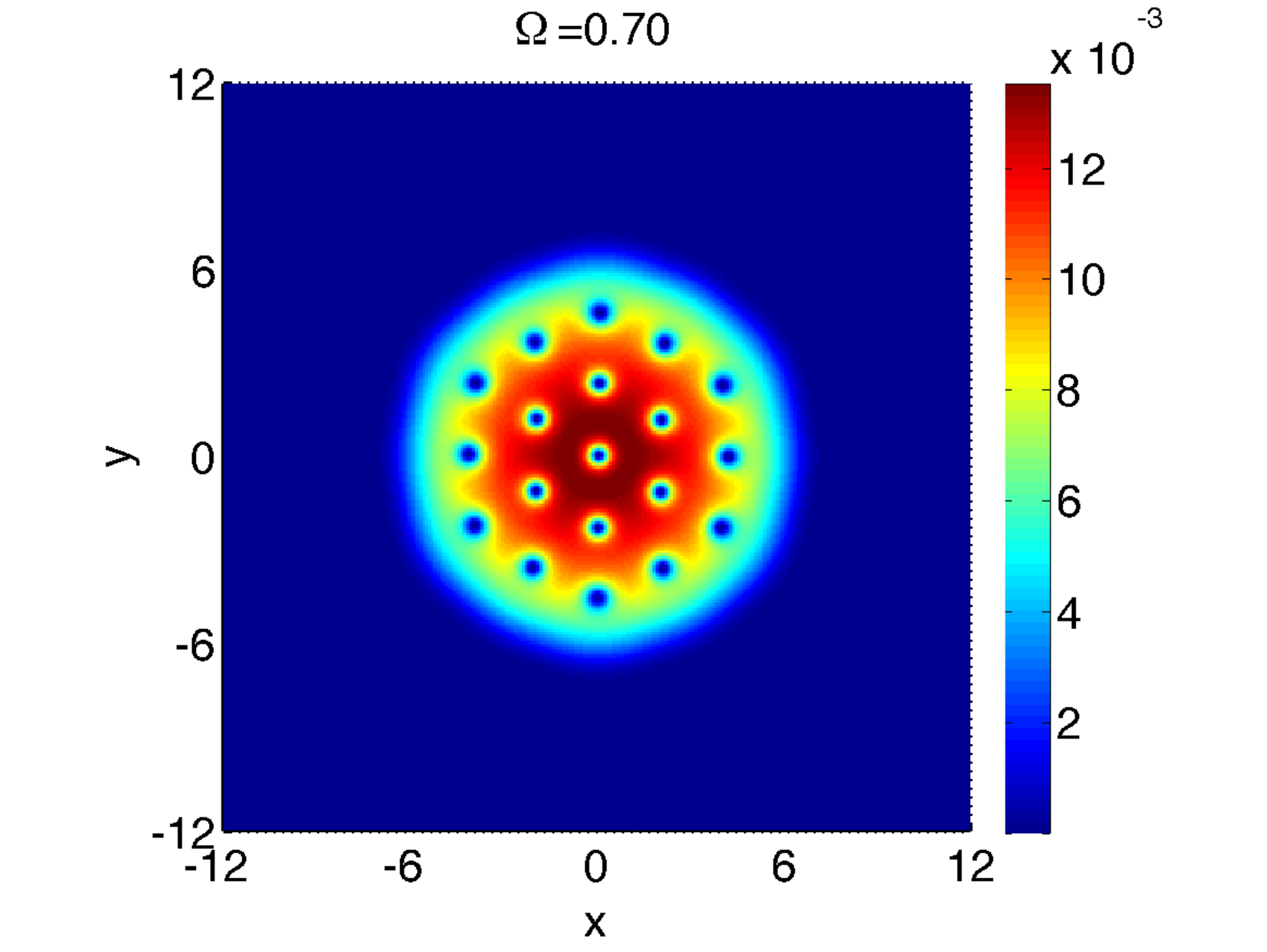}
 \includegraphics[width=0.4\textwidth]{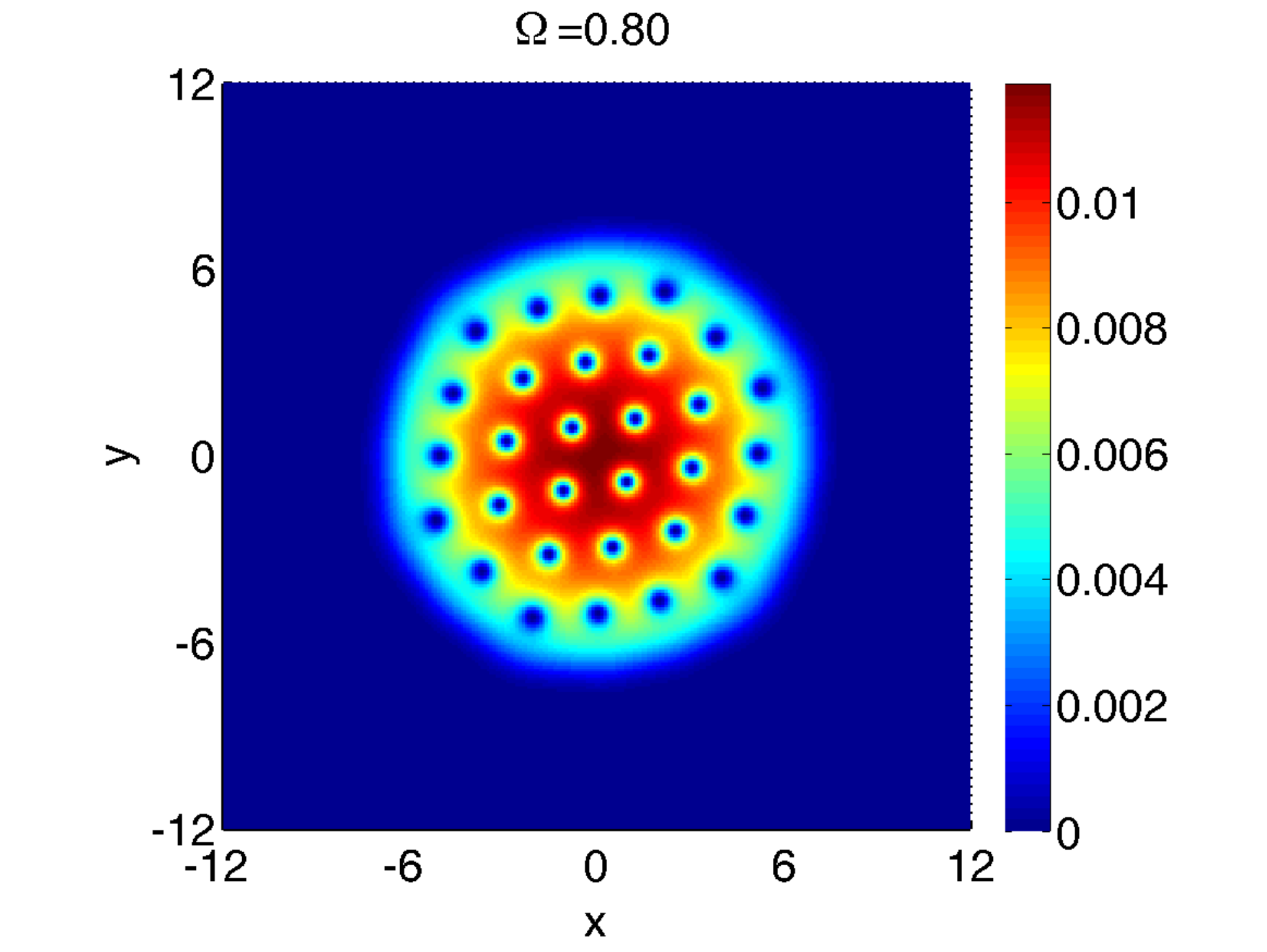}\\
 \includegraphics[width=0.4\textwidth]{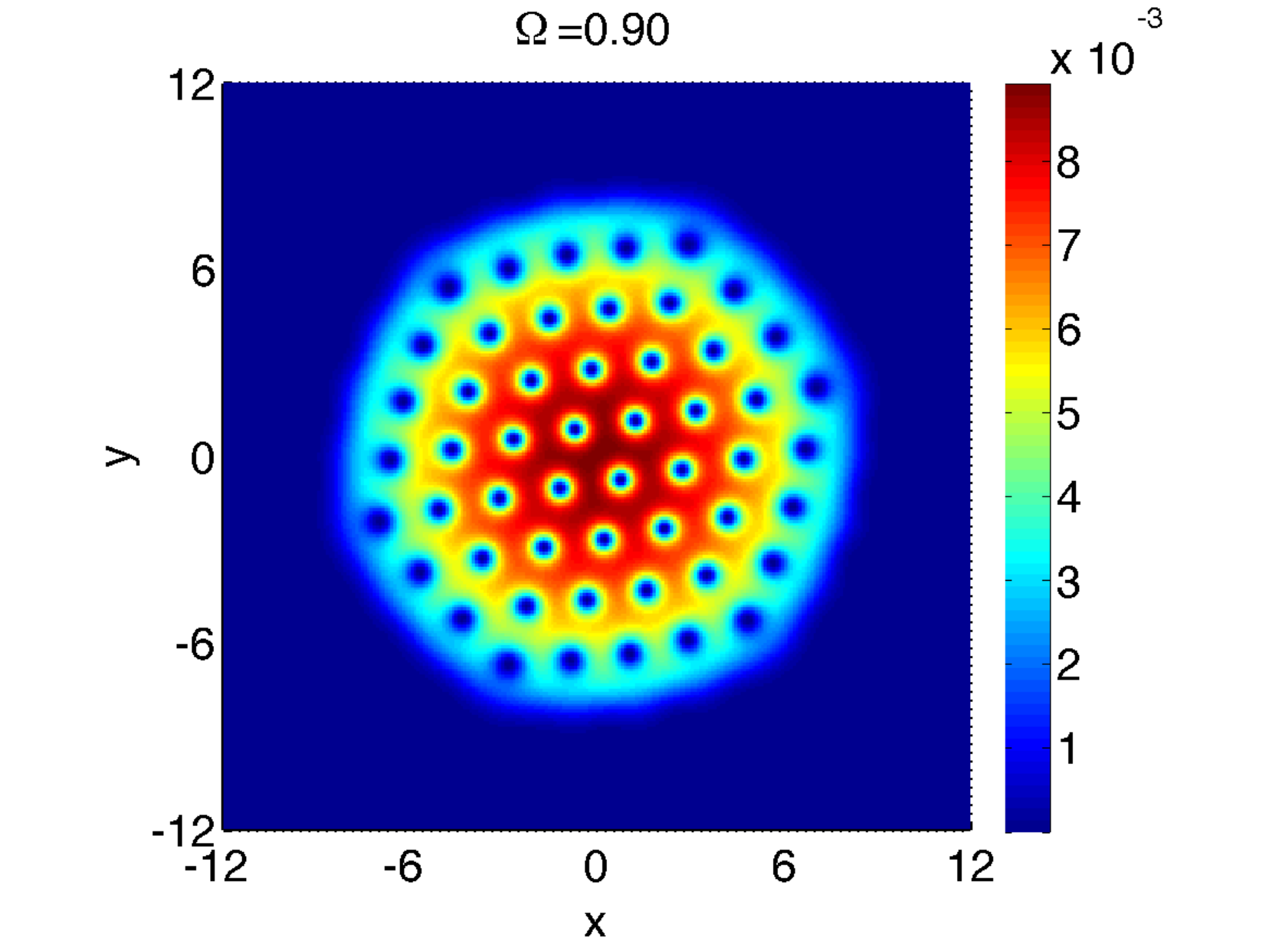}
 \includegraphics[width=0.4\textwidth]{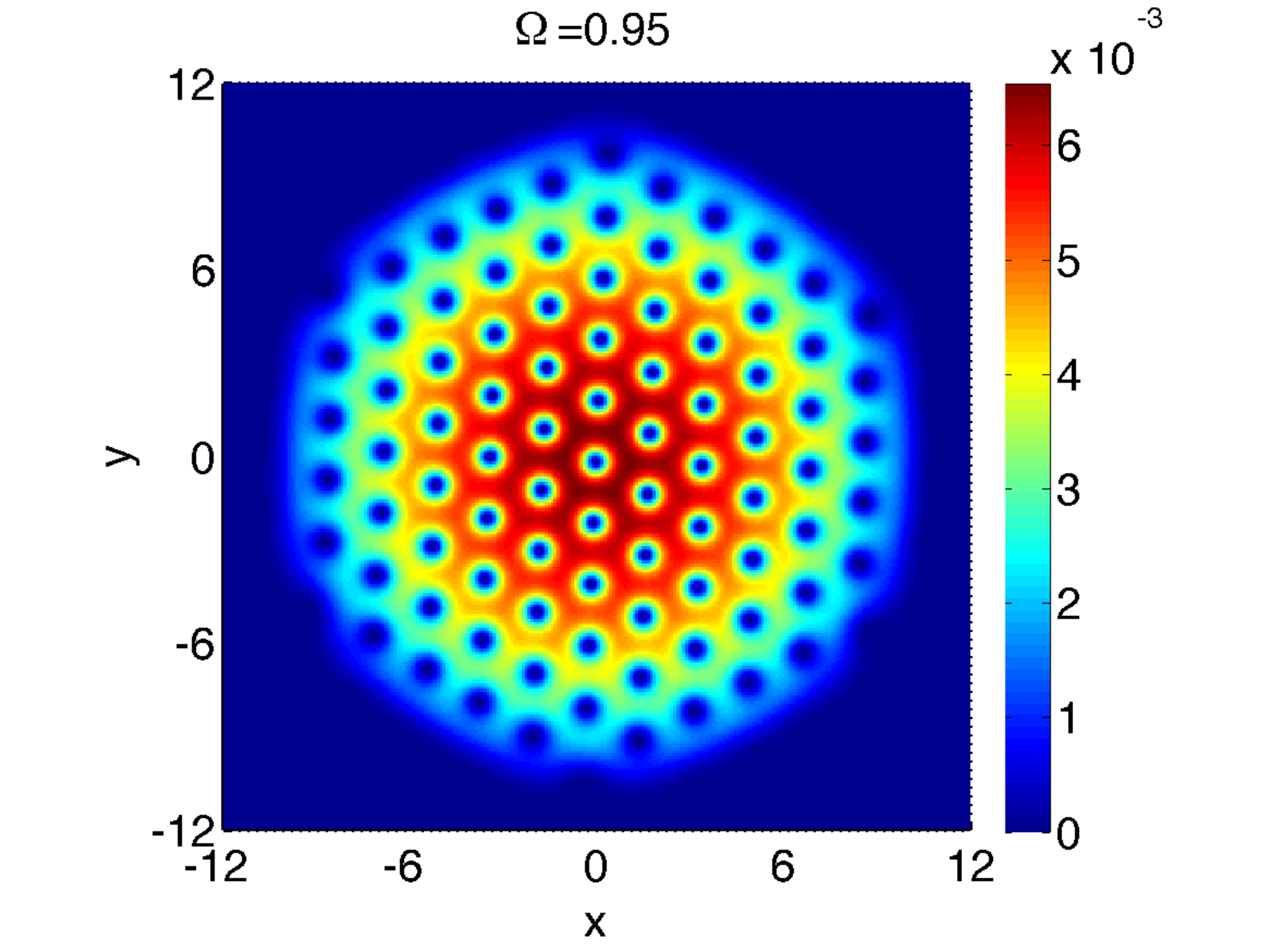}
 \end{minipage}
\end{figure}

From Tables \ref{tab-fp2d-Rot-Bet-500-energy}-\ref{tab-fp2d-Rot-Bet-1000-iter-cpu},
among those different initial data,  either (d) or ($\bar{\rm d}$) gives the lowest energy in most cases.
Thus, in practical computations, we recommend to choose either (d) or ($\bar{\rm d}$) as the initial data.
Also, it is observed that the regularized Newton algorithm converges quickly
to the stationary solution
within very few iterations, even for strong interaction, i.e., $\beta\gg1$, and fast rotation
i.e., $\Omega$ is near $1$. Compared with the normalized gradient flow method via BEFD or BESP discretization
for computing ground state of a rotating BEC
%\cite{GPELab-2014,AD,Aft2003,Bao-Cai-2013-review,Bao-Wang-Markowich-2005},
\cite{AD,Aft2003,GPELab-2014,Bao-Cai-2013-review,Bao-Wang-Markowich-2005},
the regularized Newton algorithm
significantly reduces the computational time.

%====================%====================
%\subsection{Applications to first excited state solutions}
\subsection{Application to compute asymmetric excited states}
When the trapping potential $V(\xb)$ in (\ref{prob-min}) is symmetric and the BEC is non-rotating,
similarly to those numerical methods presented in the literatures
%\cite{Bao-Cai-2013-review,Bao-Du-2004,Bao-Tang-2003,Bao-Chern-Lim-2006},
\cite{Bao-Cai-2013-review,Bao-Chern-Lim-2006,Bao-Du-2004,Bao-Tang-2003},
our numerical methods can also be applied to compute the asymmetric excited states
provided that the initial data is chosen as an asymmetric function.
To demonstrate this, we take $d=2$, $\Omega=0$ and $\beta=500$ in (\ref{prob-min}) and
the trapping potential is chosen as a combined harmonic and optical lattice potential
\be\label{asypon}
V(x,y) = \frac12\left(x^2+y^2\right) + 50\left[ \sin^2\left(\frac{\pi x}{4}\right)
+ \sin^2\left(\frac{\pi y}{4}\right) \right].
\ee
The ground and asymmetric states are numerically computed by the Algorithm \ref{alg:ConOptM}
via the SP discretization
on the bounded computational domain $U=(-16,16)^2$ which is partitioned uniformly
with the number of nodes $N_1=N_2=2^8+1$ in each direction.
The initial data is chosen as the TF approximation
\eqref{TFA} for computing the ground state $\phi_g$,
as $\phi_0(x,y)=\frac{\sqrt{2}x}{\pi^{1/2}}e^{-(x^2+y^2)/2}$ for the asymmetric excited state in  the
$x$-direction $\phi_{10}$,
as $\phi_0(x,y)=\frac{\sqrt{2}y}{\pi^{1/2}}e^{-(x^2+y^2)/2}$ for the asymmetric
excited state in the $y$-direction $\phi_{01}$,
and as $\phi_0(x,y)=\frac{2xy}{\pi^{1/2}}e^{-(x^2+y^2)/2}$ for the asymmetric excited state
in both $x$- and $y$-directions $\phi_{11}$, respectively.
The stopping criterion is set to the default value.
Table \ref{table-sp2d-Exc} lists different quantities of these states and computational cost
by our algorithm. In addition, Figure \ref{fig-sp2d-Exc} shows contour plots of
these states.

% Tab. 2D, Excited States (\epsilon_x = 1e-6)
\begin{table}[htdp]\caption{Different quantities of the ground and asymmetric excited states
and the corresponding computational cost for a BEC in 2D with
the  potential  (\ref{asypon}) and $\beta=500$ in \S 4.4.}
\label{table-sp2d-Exc}
\setlength{\tabcolsep}{6pt}
\begin{center}
\begin{tabular}{cccccccccc}\hline
$\phi$ & $\max|\phi|^2$  &  $E(\phi)$  &  $\mu(\phi)$  &  $x_{\mathrm{rms}}$  &
$y_{\mathrm{rms}}$  & iter & nfe & cpu\,(s) \\ \hline
$\phi_g$  & 0.0820 & 32.2079 & 41.7854 & 2.9851 & 2.9851 & 365 & 380 & 3.99 \\ %\hline
$\phi_{10}$ & 0.0746 & 34.6053 & 43.8248 & 3.3029 & 2.8741 & 285 & 301 & 3.18 \\ %\hline
$\phi_{01}$ & 0.3749 & 34.6053 & 43.8248 & 2.8741 & 3.3029 & 272 & 288 & 3.03 \\ %\hline
$\phi_{11}$ & 0.0666 & 37.0864 & 46.1442  & 3.1434 & 3.1434 & 117 & 125 & 1.32 \\ \hline
\end{tabular}
\end{center}
\end{table}

% Fig. 2D, Excited States
\begin{figure}[htdp]\caption{Contour plots of the ground state $\phi_g$ (a), asymmetric excited state in the $x$-direction $\phi_{10}$ (b), (c) asymmetric excited state in the $y$-direction $\phi_{01}$ (c), and
asymmetric excited state in both $x$- and $y$- directions $\phi_{11}$ (d) of a BEC in 2D with
the  potential  (\ref{asypon}) and $\beta=500$ in \S 4.4.}
\label{fig-sp2d-Exc}
\begin{minipage}[t]{1\textwidth}
\centering
\subfigure[]{
\begin{minipage}[t]{0.45\textwidth}
\centering
 \includegraphics[width=0.95\textwidth, height =
 0.8\textwidth]{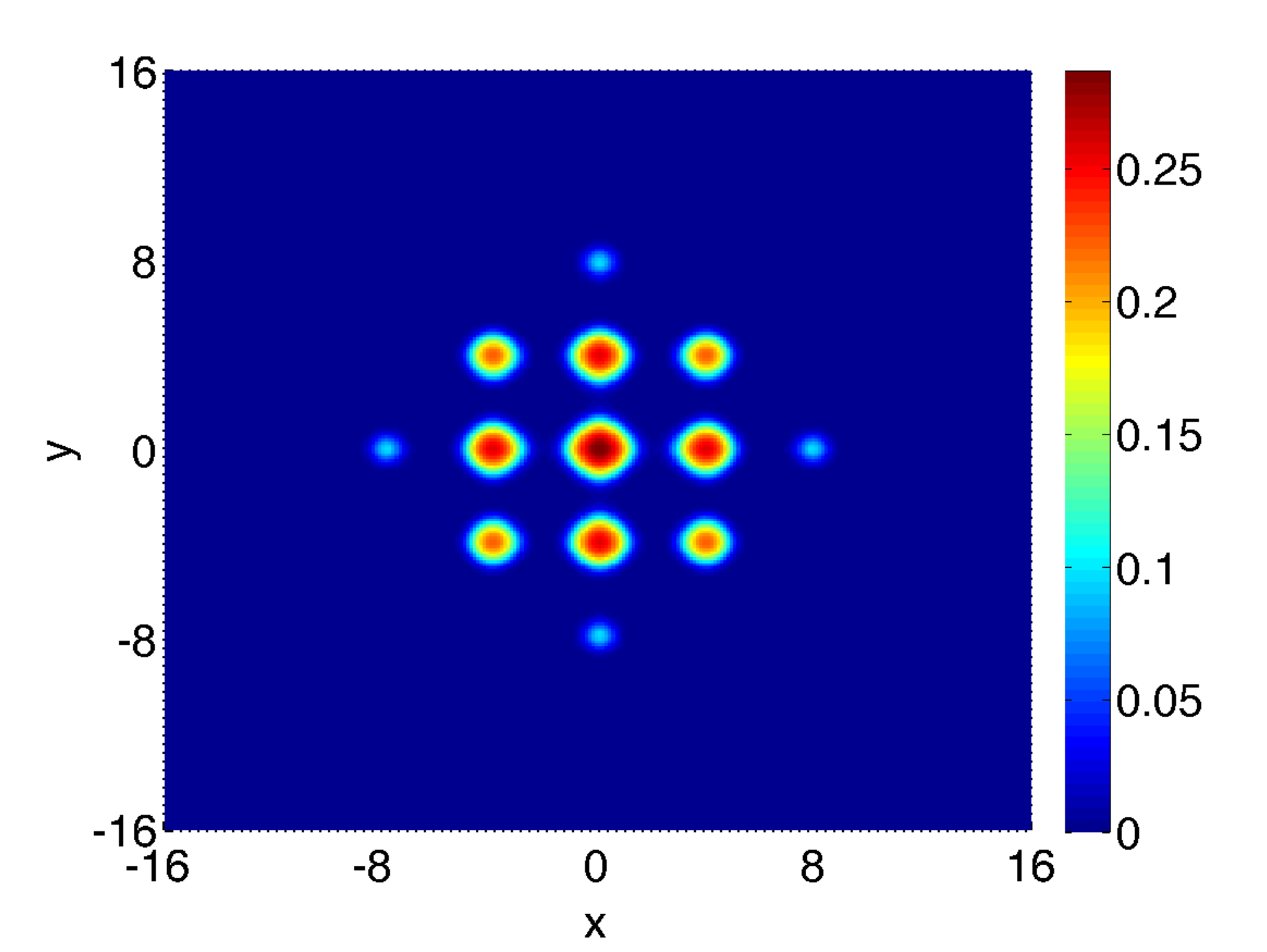}
\end{minipage}}
\subfigure[]{
\begin{minipage}[t]{0.45\textwidth}
\centering
 \includegraphics[width=0.95\textwidth, height =
 0.8\textwidth]{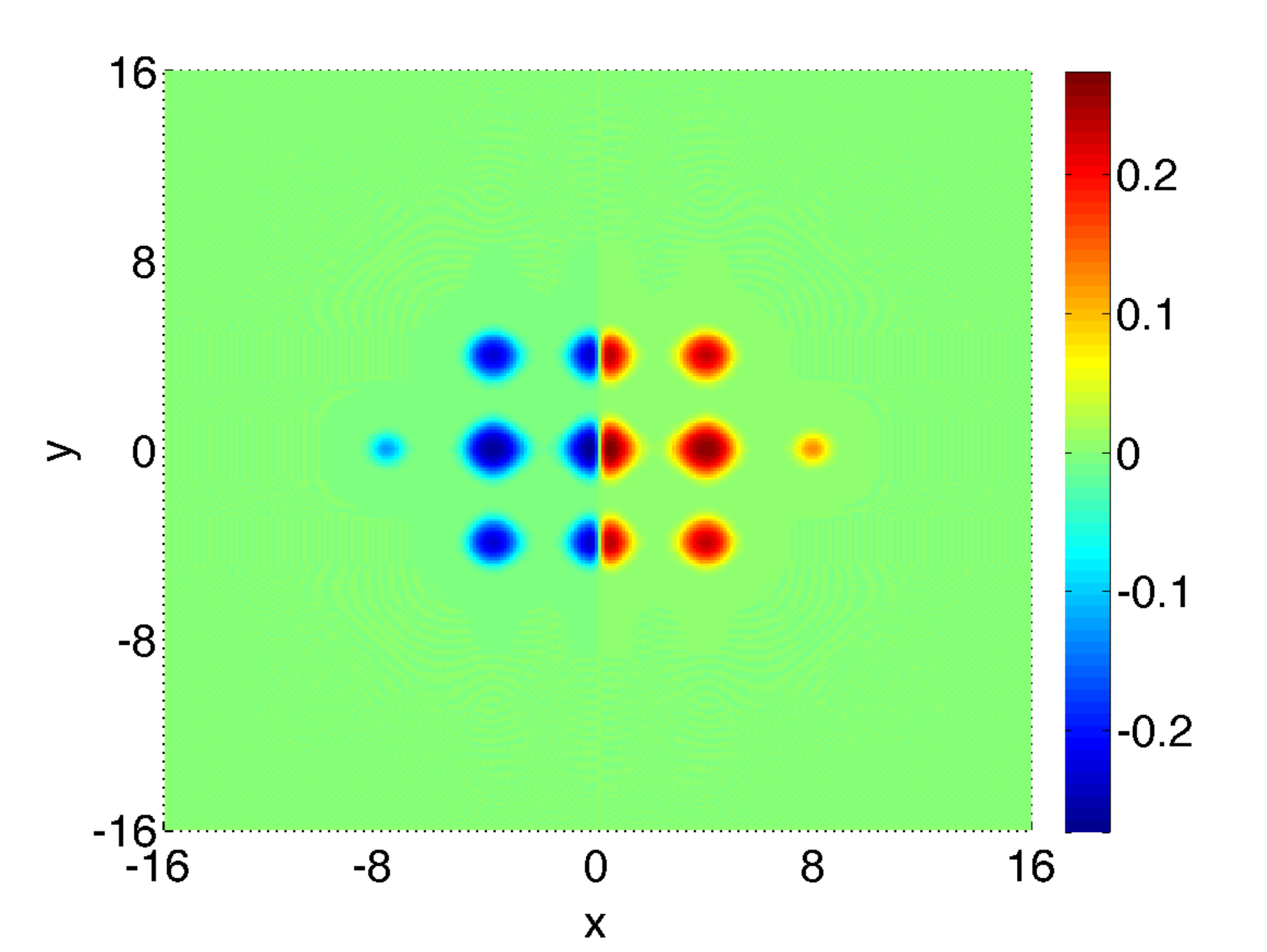}
\end{minipage}}
\subfigure[]{
\begin{minipage}[t]{0.45\textwidth}
\centering
 \includegraphics[width=0.95\textwidth, height =
 0.8\textwidth]{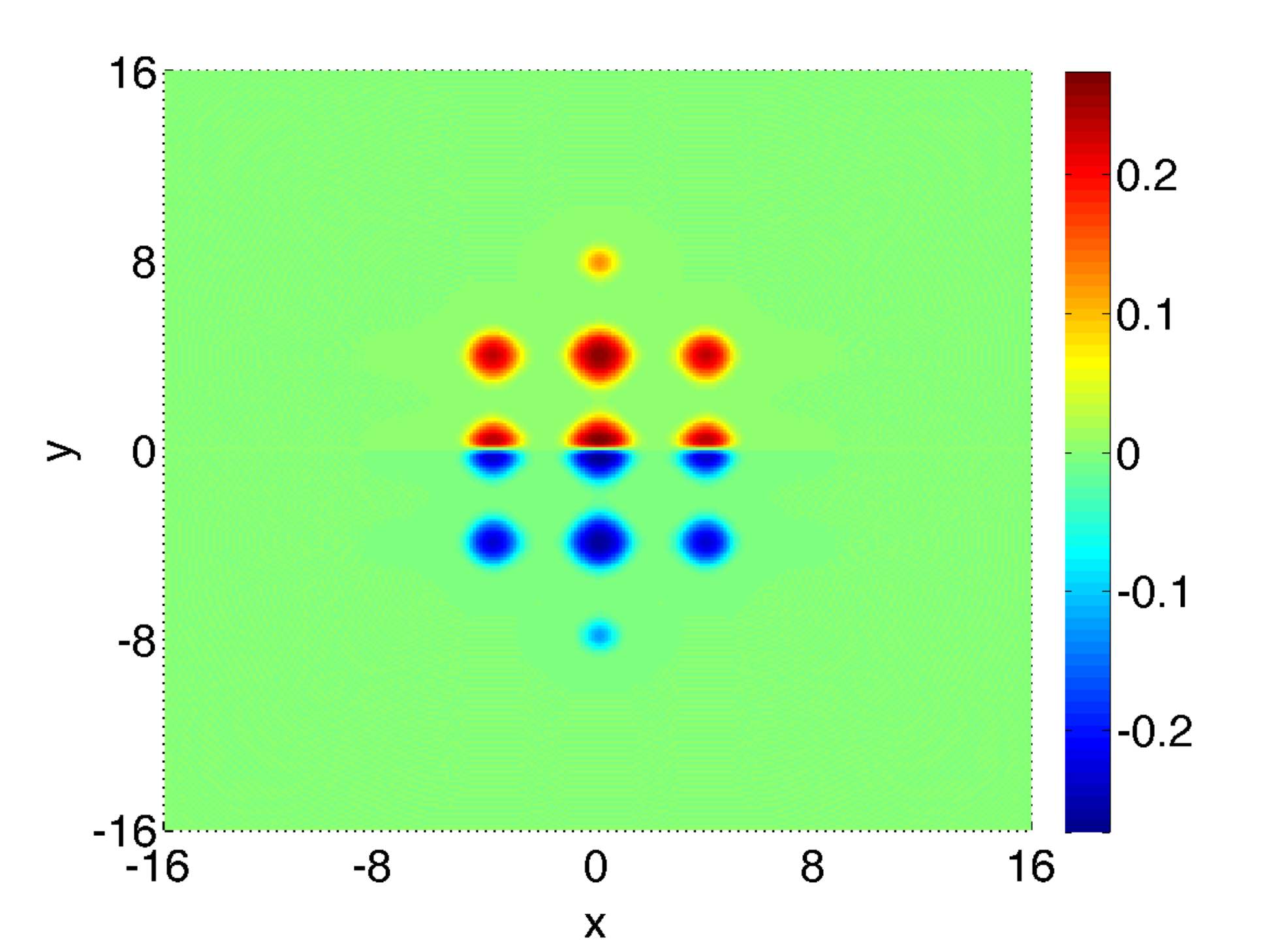}
\end{minipage}}
\subfigure[]{
\begin{minipage}[t]{0.45\textwidth}
\centering
 \includegraphics[width=0.95\textwidth, height =
 0.8\textwidth]{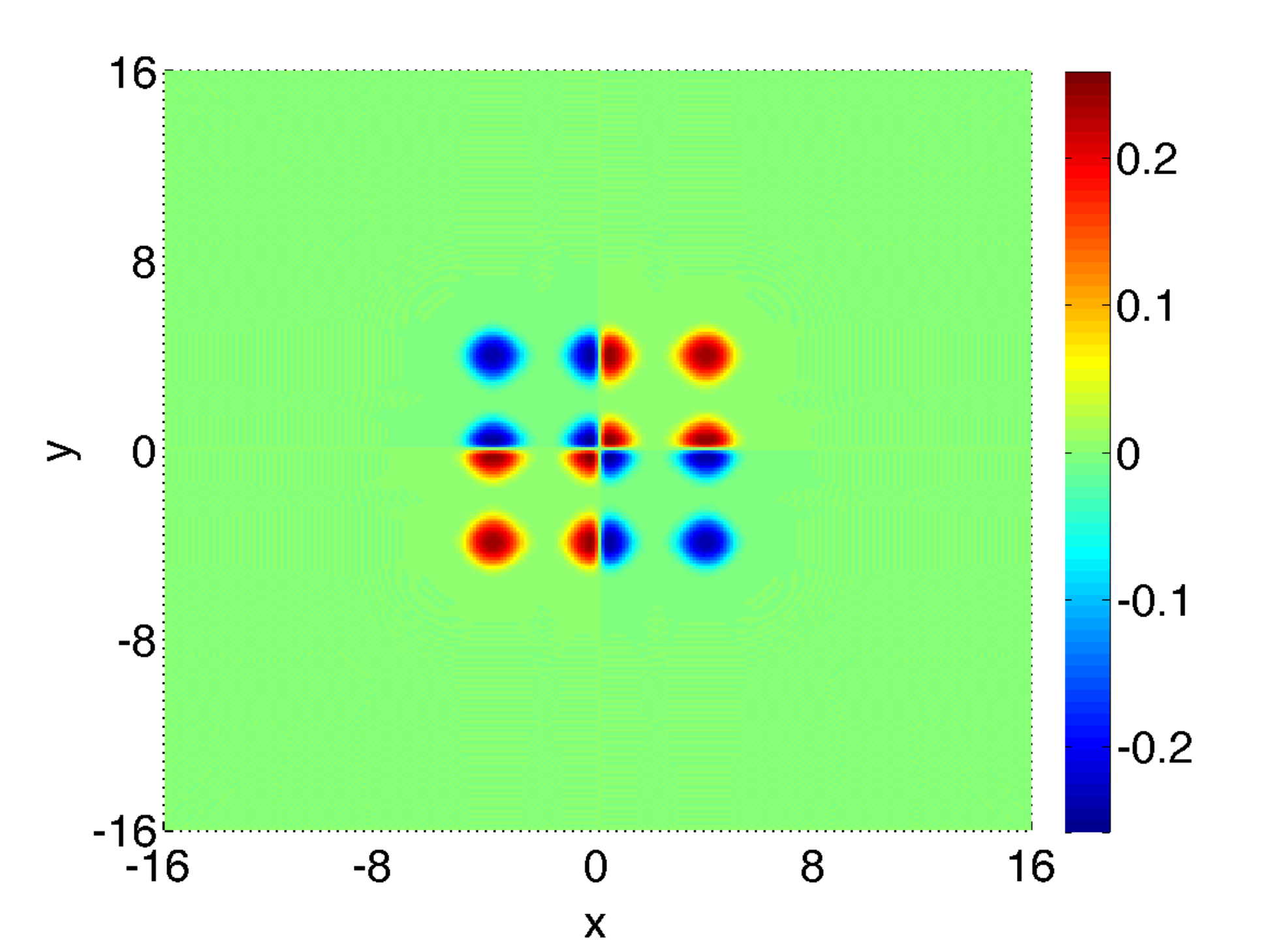}
\end{minipage}}
\end{minipage}
\end{figure}

From Table \ref{table-sp2d-Exc} and Figure \ref{fig-sp2d-Exc},
we can see that our algorithm can be used to compute the asymmetric excited states
provided that the initial data is taken as asymmetric functions.
The numerical results from our algorithm agree very well with
those reported in the literatures
%\cite{Bao-Cai-2013-review,Bao-Du-2004,Bao-Tang-2003,Bao-Chern-Lim-2006}.
\cite{Bao-Cai-2013-review,Bao-Chern-Lim-2006,Bao-Du-2004,Bao-Tang-2003}.
However, our algorithm is much faster than those methods
in the literatures
%\cite{Bao-Cai-2013-review,Bao-Du-2004,Bao-Tang-2003,Bao-Chern-Lim-2006}
\cite{Bao-Cai-2013-review,Bao-Chern-Lim-2006,Bao-Du-2004,Bao-Tang-2003}
for computing the asymmetric excited states.

%====================%====================%====================
\section{Concluding remarks}\label{conclusion}

Different spatial discretizations including the finite difference method, sine pesudospectral
and Fourier pseudospectral methods were adopted to
discretize the energy functional and constraint for
computing the ground state of Bose-Einstein condensation (BEC).
Then the original infinitely dimensional constrained minimization
problem was reduced to a finite dimensional minimization problem
with a spherical constraint.
A regularized Newton method was proposed by using a feasible gradient type
method as an initial approximation and solving a standard trust-region subproblem
obtained from approximating the energy
functional by its second-order Taylor expansion
with a regularized term at each
Newton iteration as well as adopting a cascadic multigrid
technique for selecting initial data.
The convergence of the method was established by the standard optimization theory.
Extensive numerical examples of non-rotating BEC in 1D and 3D and rotating BEC in 2D with different
trapping potentials and parameter regimes
demonstrated the efficiency and accuracy as well as
 robustness of our method.
Comparison to existing numerical methods in the literatures showed that our numerical
method is significantly faster than those methods proposed in the literatures
for computing ground states of BEC.

%\setcitestyle{numbers}
%\bibsep=0.1cm
%\bibliographystyle{plainnat}
%\bibliographystyle{plain}
%\bibliographystyle{siam}
%\bibliography{optimization,BEC}

\end{document}